\documentclass{article}

\usepackage[a4paper,centering,top=1in,bottom=1in,footskip=.5in]{geometry}

\usepackage{datetime}

\usepackage{tikz}

\usepackage{eucal}

\usepackage[all]{xy}

\usepackage{amssymb, amsmath, amsthm}

\usepackage[shortlabels]{enumitem}

\numberwithin{equation}{section}

\newlength{\bibitemsep}
\newlength{\bibparskip}
\let\oldthebibliography\thebibliography
\renewcommand\thebibliography[1]{%
  \oldthebibliography{#1}%
  \setlength{\parskip}{\bibitemsep}%
  \setlength{\itemsep}{\bibparskip}%
}

\usepackage[pdftex]{hyperref}

\hypersetup{
  colorlinks = true,
  linkcolor = black,
  urlcolor = gray,
  citecolor = black,
}

\newtheorem{theorem}{Theorem}[section]
\newtheorem{proposition}[theorem]{Proposition}
\newtheorem{corollary}[theorem]{Corollary}
\newtheorem{lemma}[theorem]{Lemma}
\theoremstyle{definition}

\newtheorem{definition}[theorem]{Definition}
\newtheorem{example}[theorem]{Example}
\newtheorem{examples}[theorem]{Examples}

\newtheorem{remark}[theorem]{Remark}

\newcommand{\bbF}{\mathbb{F}}
\newcommand{\bbN}{\mathbb{N}}
\newcommand{\bbZ}{\mathbb{Z}}

\newcommand{\calG}{\mathcal{G}}

\newcommand{\calP}{\mathcal{P}}
\newcommand{\calQ}{\mathcal{Q}}
\newcommand{\calR}{\mathcal{R}}
\newcommand{\calS}{\mathcal{S}}
\newcommand{\calX}{\mathcal{X}}
\newcommand{\calZ}{\mathcal{Z}}

\newcommand{\rmB}{\mathrm{B}}
\newcommand{\rmC}{\mathrm{C}}

\newcommand{\rmF}{\mathrm{F}}
\newcommand{\rmK}{\mathrm{K}}
\newcommand{\rmQ}{\mathrm{Q}}

\newcommand{\rmU}{\mathrm{U}}

\newcommand{\rmZ}{\mathrm{Z}}

\newcommand{\Op}{\mathrm{Op}}
\newcommand{\Orb}{\mathrm{Orb}}
\newcommand{\Ad}{\mathrm{Ad}}
\newcommand{\Id}{\mathrm{Id}}
\newcommand{\id}{\mathrm{id}}

\newcommand{\Gr}{\mathrm{Gr}}
\newcommand{\Inn}{\mathrm{Inn}}
\newcommand{\Aut}{\mathrm{Aut}}
\newcommand{\Sym}{\mathrm{Sym}}
\newcommand{\Inj}{\mathrm{Inj}}
\newcommand{\Sur}{\mathrm{Sur}}
\newcommand{\Mor}{\mathrm{Mor}}

\newcommand{\End}{\mathrm{End}}

\newcommand{\Dih}{\mathrm{Dih}}

\newcommand{\FQ}{\mathrm{FQ}}
\newcommand{\FR}{\mathrm{FR}}

\newcommand{\primeQ}{\calQ_\mathrm{pri}}
\newcommand{\Pcon}{\calP_\mathrm{con}}
\newcommand{\Qcon}{\calQ_\mathrm{con}}
\newcommand{\Rcon}{\calR_\mathrm{con}}

\renewcommand{\phi}{\varphi}
\renewcommand{\epsilon}{\varepsilon}

\renewcommand{\emptyset}{\text{\upshape\O}}

\DeclareMathOperator{\lcm}{lcm}

\begin{document}

\thispagestyle{empty}

\title{\bf Burnside rings for \hbox{racks and quandles}}

\author{Nadia Mazza and Markus Szymik}

\newdateformat{mydate}{\monthname~\twodigit{\THEYEAR}}
\date{\mydate\today}

\maketitle


We restructure and advance the classification theory of finite racks and quandles by employing powerful methods from transformation groups and representation theory, especially Burnside rings. These rings serve as universal receptacles for those invariants of racks and quandles that are additive with respect to decompositions. We present several fundamental results regarding their structure, including additive bases and multiplicative generators. We also develop a theory of marks, which is analogous to counting fixed points of group actions and computing traces in character theory, and which is comprehensive enough to distinguish different elements in the Burnside rings. The new structures not only offer a fresh framework for the classification theory of finite racks and quandles but also equip us with tools to develop these ideas and create interfaces that strengthen connections with related areas of algebra. For example, they extend the Dress--Siebeneicher theory of the Burnside ring of the infinite cyclic group beyond the realm of permutation racks.


\section*{Introduction}

Racks and quandles are algebraic structures that model the axioms of knotting and braiding. Perhaps because of their fundamental role in abstracting those patterns, these structures have been rediscovered many times and given many names. 
The subject dates back at least to Peirce's~1880 work on the algebra of logic~\cite{Peirce}, but we will resist the temptation to trace its history here, as such attempts are often inaccurate. For instance, it is seldom mentioned that Burde's~1978 survey~\cite{Burde} had described the algebra of knot diagrams well before Joyce~\cite{Joyce} and Matveev~\cite{Matveev} built on Waldhausen's results~\cite{Waldhausen} and explained how to use it to classify knots in terms of this algebra. We hasten to point out, though, that knot theory is only a small part of the mathematics that has profited from the introduction of these concepts. In fact, racks and quandles are more general than groups, which are omnipresent in mathematics and give rise to racks and quandles through conjugation. What motivates us is not the increased generality, but the evidence that has shown racks and quandles to be relevant to researchers working in many different contexts. For instance, in geometry, in addition to knot theory as mentioned above, these structures can already be observed in the elementary context of the intercept theorem, and have appeared in Brieskorn's work~\cite{Brieskorn} on singularities, Yetter's work~\cite{Yetter} on monodromy, and---relatedly---moduli spaces of branched covers~\cite{EVW, Randal-Williams}. 

Even from a purely algebraic perspective, it is hard to ignore the influence the theory of racks and quandles has had on us over the past decades. While abelian quandles are~`just' modules over the Laurent polynomial ring, the general nonabelian theory impacts the Yang--Baxter equations~\cite{CES, Eisermann, Lebed--Vendramin}, Hopf algebras~\cite{Andruskiewitsch--Grana, Vendramin, HLV, Wagemann, Kashaev, Andruskiewitsch, Carnovale--Maret}, Leibniz algebras~\cite{Loday,Kinyon,Krahmer--Wagemann}), and---more recently---also various aspects of number theory~\cite{Takahashi, Shusterman, Szymik:AS, DS}. Many of these applications require a thorough understanding of finite racks and quandles, and even situations that typically require infinite structures can reasonably be studied through finite approximations: Bardakov, Singh, and Singh have shown that free quandles and knot quandles are residually finite~\cite{BSS}. The purpose of this paper is to advance our understanding of finite racks and quandles in two ways.

First, we argue that Burnside rings, which have been profitably employed to manage the stocktaking and bookkeeping in transformation groups and representation theory~(see~\cite{Solomon, Dress, tomDieck:1, tomDieck:2, Bouc:Burnside}), can be adapted to the new context of racks and quandles. This adaptation offers enhanced methods for reasoning about the classification problem, formulating questions, and presenting answers. We note that this adaptation is not straightforward. Recall that the Burnside ring~$\rmB(G)$ of a group~$G$ is generated by classes represented by finite~$G$--sets, and that the addition is induced by their disjoint union, which is also the categorical sum~(or coproduct) in the category of~$G$--sets. For racks and quandles, coproducts are no longer modelled by disjoint unions. While both of these constructions give rise to symmetric monoidal categories, neither of them leads to a useful notion of a Burnside ring for finite racks and quandles. Instead, we replace these compositional structures with an approach that emphasises decompositions of racks and quandles. 

Second, once we have established our framework, we prove fundamental results about the Burnside rings of finite racks and quandles and relate them to existing structures that are important in algebra, such as in transformation groups and representation theory. Part of our work can also be thought of as an extension of the Dress--Siebeneicher theory~\cite{Dress--Siebeneicher:pro-finite, Dress--Siebeneicher:Witt} of the Burnside ring of the infinite cyclic group~(and its relations to~$\lambda$--rings and the universal ring of Witt vectors) beyond the realm of permutation racks.


We define the Burnside ring~$\rmB(\calR)$ of finite racks so that it is the universal additive invariant for finite racks~(see Definition~\ref{def:B(R)}). This means that every finite rack~$R$ defines an element~$b(R)$ in it, and we have a relation~$b(R)=b(S)+b(T)$ whenever we can decompose~$R$ into subracks~$S$ and~$T$ in a way we make precise in Definition~\ref{def:decomposition}. The product is defined so that we have~\hbox{$b(R\times S)=b(R)b(S)$}~(see Proposition~\ref{prop:AR_multiplication}). By construction, the ring~$\rmB(\calR)$ is additively generated by the classes~$b(R)$ of finite racks, but it turns out that it is generated by the classes of the \textit{connected} racks. Moreover, we show that~$b(C)=b(D)$ for connected racks~$C$ and~$D$ if and only if~$C$ and~$D$ are isomorphic~(see Theorem~\ref{thm:inj_on_connected}), which is not true in general. Furthermore, there are no additive relations between the classes of the connected racks, so that these form an integral basis of the Burnside ring~$\rmB(\calR)$ as an abelian group~(see Theorem~\ref{thm:basis}). We note that we could have used this result to \textit{define} the Burnside ring; then, however, it would have been our duty to show that this definition gives, indeed, also the universal additive invariant.

The Burnside ring~$\rmB(\calR)$ of finite racks contains, as subrings with retractions, the Burnside ring~$\rmB(\calQ)$ of finite quandles, which is similarly defined, and the Burnside ring~$\rmB(\bbZ)$ of the group of integers, which is isomorphic to the Dress--Siebeneicher Burnside ring of finite permutation racks~(see Proposition~\ref{prop:quandles_split}, Example~\ref{ex:another_retraction}, and Proposition~\ref{prop:permutations_split} for the retractions). We show that there are no multiplicative relations between these two subrings, so that their tensor product~\hbox{$\rmB(\calQ)\otimes\rmB(\bbZ)$} is a subring of~$\rmB(\calR)$. However, the latter is bigger~(see Corollary~\ref{cor:tensor_embed} to Theorem~\ref{thm:CxQ<R} and Example~\ref{ex:not_surjective}), and its multiplicative structure is much more difficult to understand. In contrast, we can show that the Burnside ring~$\rmB(\calQ)$ of finite quandles is a monoid ring~(see Theorem~\ref{thm:Burnside_for_quandles_is_monoidal}). Thus, we can speak of \textit{prime} quandles, which generate the Burnside ring~$\rmB(\calQ)$ of finite quandles multiplicatively~(see Definition~\ref{def:prime} and Theorem~\ref{thm:multiplicative_generators}). There is a clear analogy with prime numbers and prime knots, and the~`arithmetic' of connected quandles remains to be explored; we can only speculate about the uniqueness of prime factorisations here~(see Remark~\ref{rem:uniqueness}). We are not aware of any multiplicative relations between prime quandles, and we can prove a cancellation result~(see Theorem~\ref{thm:cancellation}) that implies that none of the classes~$b(Q)$ of non-empty finite quandles~$Q$ is a zero-divisor in~$\rmB(\calQ)$. The proof of this and many other of our results is facilitated by our theory of marks.

Our theory of marks is inspired by analogies with classical ideas from transformation groups and representation theory. Recall that the marks of the Burnside ring~$\rmB(G)$ of a group~$G$ are ring homomorphisms~$\rmB(G)\to\bbZ$ that give the numbers of~$H$--fixed points for the various subgroups~$H\leqslant G$. Together, they embed the Burnside ring~$\rmB(G)$ into a product of copies of the ring~$\bbZ$, just like the characters of a finite group embed the representation ring as a lattice in a ring of functions. In our context, it turns out to be a fruitful idea to count morphisms from finitely generated connected racks~$C$ into racks, and this leads to the notion of marks~$\Phi_C\colon\rmB(\calR)\to\bbZ$ for finite racks. Note that every knot gives rise to such a mark~(see Remark~\ref{rem:knots}), and we can show that marks are strong enough to distinguish all elements in the Burnside ring~(see Theorem~\ref{thm:marks}).

Another relation to transformation groups and representation theory is not just an effective analogy: there is an isomorphism between~$\rmB(G)$ and a suitably defined crossed Burnside ring of group actions. Recall that the crossed Burnside ring~$\rmB^\times(G)$ is made from~\textit{crossed}~$G$--sets, which are~$G$--sets that are equipped with a map to the~$G$--set~$\Ad(G)$, which is the set underlying~$G$ with the~$G$--action given by conjugation~\cite{Oda--Yoshida, Bouc:2003b, Mazza}. Such crossed~$G$--sets define racks, and every rack comes from such a crossed~$G$--set. As one way to make the relationship precise, we introduce a category~$\calX$ of \textit{crossed actions} and show that we have an isomorphism~$\rmB(\calR)\cong\rmB(\calX)$ of rings~(see Theorem~\ref{thm:iso_BX_BR}).


Here is an outline of the paper. In Section~\ref{sec:racks}, we include the definitions and notation related to racks and groups acting on them that we employ throughout the paper. Section~\ref{sec:decompositions} reviews decompositions of racks. The Burnside rings of racks and quandles, the primary focus of this article, are introduced in Section~\ref{sec:Burnside}, which also encompasses our discussion of additive invariants and their relation to Grothendieck rings. Section~\ref{sec:basis} addresses the additive structure of our Burnside rings. It contains our proof that the connected objects provide an integral basis. The following Section~\ref{sec:marks} presents the definitions and examples of marks, along with a proof showing that these are sufficient to distinguish elements in the Burnside rings. We discuss permutation racks in Section~\ref{sec:perm} and quandles in Section~\ref{sec:prime}. To describe multiplicative generators of the Burnside ring of quandles, we also consider product factorisations of quandles. In the final Section~\ref{sec:crossed}, we discuss the relationship between our rings, the Burnside rings of crossed~$G$--sets for a fixed group~$G$, and a~`global' variant thereof, which does not fix a group.


\subsection*{Acknowledgement}

The second author thanks the Isaac Newton Institute for Mathematical Sciences, Cambridge, for support and hospitality during the programme~`Equivariant Homotopy Theory in Context', where work on this paper was undertaken. This work was supported by EPSRC grant EP/Z000580/1.


\section{Racks, quandles, and groups acting on them}\label{sec:racks}

This section contains the introductory material on racks and actions of groups on them that we need further on. We also use the occasion to fix our notation for the rest of the paper. Our basic references are~\cite{Joyce, Brieskorn, Fenn--Rourke, Andruskiewitsch--Grana}.


\subsection{Racks}

We start with a definition of the objects of interest.

\begin{definition}
A~\textit{rack}~$(R,\rhd)$ is a set~$R$ together with a binary operation~\hbox{$\rhd\colon R\times R\to R$} such that the left multiplications~\hbox{$\ell_a\colon b\mapsto a\rhd b$} are automorphisms of the pair~$(R,\rhd)$ for all~$a$. 
Automorphisms, of course, are bijective morphisms, and the fact that~$\ell_a$ is a morphism can be rewritten as~\hbox{$\ell_a(b\rhd c)=\ell_a(b)\rhd\ell_a(c)$}, or~\hbox{$a\rhd(b\rhd c)=(a\rhd b)\rhd(a\rhd c)$} for all~$b$ and~$c$. 
\end{definition}

\begin{remark}\label{rem:sums}
There is an obvious category~$\calR$ of racks, and it has all limits and colimits. The limits are preserved by the forgetful functor to the category of sets. In particular, the product~\hbox{$R\times S$} of two racks~$R$ and~$S$ is supported on the cartesian product of the underlying sets, and the rack operation is component-wise. Colimits, however, are more complicated. In particular, the sum~$R+S$ of two racks is, in general, not supported on the disjoint union of the underlying sets, and the sum of two finite racks need not be finite. The empty set is a rack: it is an initial object in the category of racks, as the theory of racks has no constants.
\end{remark}

\begin{example}\label{ex:permutation_racks}
A~\textit{permutation rack} is a rack~$(R,\pi)$, where~$\pi\in\Sym(R)$ is a permutation of the set~$R$, and we have~\hbox{$a\rhd b=\pi(b)$} for all~\hbox{$a,b\in R$}. In particular, for~$\pi=\id_R$, we get the~\textit{trivial racks}, which satisfy~\hbox{$a\rhd b=b$} for all~\hbox{$a,b\in R$}.
\end{example}

For some purposes, it is helpful to think of racks as generalisations of permutations.


\subsection{Quandles}

By definition, each rack comes with its own supply of symmetries: the left-multiplications. However, these symmetries are not natural in the categorical sense. Recall that a natural transformation~$\Id_\calR\to\Id_\calR$ from the identity functor to itself is a family~$\phi_R\colon R\to R$ of automorphisms, one for each rack~$R$, such that for all morphisms~\hbox{$f\colon R\to S$} of racks we have~$\phi_S\circ f=f\circ \phi_R$. These natural symmetries form the~\textit{centre} of the category of racks, and they have been determined in~\cite[Thm.~5.4]{Szymik:2018}: the centre~$\rmZ(\calR)=\End(\Id_\calR)$ of the category of racks is an abelian monoid isomorphic to~$\bbZ$ under addition, where the generator corresponds to the~\textit{canonical automorphism}
\begin{equation}\label{eq:canonical_automorhphism}
a\mapsto a\rhd a=\sigma(a).
\end{equation}

\begin{definition}
Racks that do not have natural symmetries other than the identity are called~\textit{quandles}; they satisfy~$a\rhd a=a$ for all~$a$. 
\end{definition}

We write~$\calQ$ for the full subcategory of the category~$\calR$ of racks that consists of the quandles. 

\begin{example}
The only permutation racks that are quandles are the trivial racks~(see Example~\ref{ex:permutation_racks}).
\end{example}

\begin{examples}
Any group~$G$ defines a quandle with respect to the operation~\hbox{$g\rhd h=ghg^{-1}$}. More generally, the same formula works for any subset~\hbox{$Q\subseteq G$} that is a union of conjugacy classes.
\end{examples}

\begin{examples}\label{ex:free}
A rack~$R$ is called~\textit{free} if there exists a subset~$B\subseteq R$ such that any map~\hbox{$B\to S$} of sets into a rack~$S$ extends uniquely into a rack morphism~$R\to S$. The free rack~$\FR(B)$ on a set~$B$ can be modelled on the set~\hbox{$\FR(B)=\rmF(B)\times B$}, where~$\rmF(B)$ is the free group on~$B$. The map~\hbox{$B\to\FR(B)$} is given by~$b\mapsto(e,b)$, where~$e$ denotes the identity element of the group, and the rack structure is given by~\hbox{$(g,x)\rhd(h,y)=(gxg^{-1}h,y)$}. The product~\hbox{$\FR(B)=\rmF(B)\times B$} is never finite unless~$B$ is empty, in which case~$\FR(B)$ is empty, too. Any rack is the quotient of a free rack: given any rack~$R$, there is a canonical surjection~$\FR(R)\to R$, which is adjoint to the identity. It is surjective, as every element~$r$ is hit by its corresponding generator~$(e,r)$: the canonical surjection restricts to the identity on the subset~\hbox{$R\subseteq\FR(R)$} of generators. The free quandle on a set~$S$ can be modelled as the subset~$\FQ(S)\subseteq\rmF(S)$ of conjugates of the elements of~$S$ inside the free group on~$S$.
\end{examples}

\begin{remark}
A rack is~\textit{cyclic} if it is generated by a single element. It is easy to classify the cyclic racks. The free rack on one generator is isomorphic to~$\bbZ$ with~$a\rhd b=b+1$. This is a permutation rack and infinite cyclic, generated by any element. The finite cyclic racks are the permutation racks supported on~$\bbZ/n$, using the same formula but read in~$\bbZ/n$, so finite cyclic racks are permutation racks for a full cycle.
\end{remark}

The categories~$\calR$ of racks and~$\calQ$ of quandles are challenging to comprehend. This observation is underlined by the fact, which we will prove in a sequel~\cite{Mazza--Szymik:Mackey}, that neither the category of racks nor the category of quandles is a topos. The relevance of this result for us comes from the fact that almost all other categories in this paper~\textit{are} topoi. In fact, they are functor categories, i.e., categories of functors from some category into the category of sets, and as such, they are easier to deal with than~$\calR$ and~$\calQ$; their disjoint unions and coproducts are much better behaved than here.


\subsection{Groups acting on racks}\label{sec:group_actions}

Let~$R$ be a rack. The map~$a\mapsto\ell_a$ is a map from~$R$ to the symmetric group~$\Sym(R)$ of permutations of~$R$. It is compatible with the rack structure on both sides:~\hbox{$\ell_{a\rhd b}=\ell_a\rhd\ell_b$}. The image of this map lies in the subgroup~\hbox{$\Aut(R,\rhd)\leqslant\Sym(R)$} of automorphisms of~$(R,\rhd)$. In fact, the image lies in the subgroup~\hbox{$\Inn(R,\rhd)\leqslant\Aut(R,\rhd)$} of~\textit{inner} automorphisms of~$(R,\rhd)$, which is---by definition---the subgroup of~$\Aut(R,\rhd)$ generated by the~left-multiplications~$\ell_a$. The group~$\Inn(R,\rhd)$ is sometimes called the~\textit{operator group}~$\Op(R,\rhd)$, for example by Fenn and Rourke~\cite{Fenn--Rourke}.

Another group that is canonically acting on a rack~$R$ is the~\textit{enveloping group}~$\Gr(R,\rhd)$ of the rack. Recall that the group~$\Gr(R,\rhd)$ is generated by the set~\hbox{$\{g(x)\mid x\in R\}\cong R$} and subject to the relations~\hbox{$g(x\rhd y)=g(x)g(y)g(x)^{-1}$}. The group~$\Gr(R,\rhd)$ acts on~$R$ by~\hbox{$g(x)\cdot y=x\rhd y$} for all~$x,y\in R$, and the group~$\Gr(R,\rhd)$ is universal with respect to rack morphisms~\hbox{$R\to G$} into groups~$G$. This means that there is a natural~(in the group~$G$ and the rack~$R$) bijection
\[
\Mor_\calG(\Gr(R,\rhd),G)\cong\Mor_\calR(R,G),
\]
where on the right-hand side, we consider~$G$ as a rack by means of its conjugation. In other words, the functor~$(R,\rhd)\mapsto\Gr(R,\rhd)$ from the category~$\calR$ of racks to the category~$\calG$ of groups is left-adjoint to the conjugation rack functor.

\begin{remark}
The enveloping group~$\Gr(R,\rhd)$ is finite only if~$R$ is empty: if~$R$ is not empty, the surjection~$R\to\star$ onto~$\star$ induces a surjection~$\Gr(R)\to\Gr(\star)=\bbZ$ of groups. Here, we denote by~$\star$ the singleton rack, which is a terminal object in the category of racks.
\end{remark}

The enveloping group~$\Gr(R,\rhd)$ of the rack~$(R,\rhd)$ surjects onto the inner automorphism group~$\Inn(R,\rhd)$ via~\hbox{$g(x)\mapsto\ell_x$}. We summarise the discussion as follows:
\[
\rmF(R)\twoheadrightarrow\Gr(R,\rhd)
\twoheadrightarrow\Inn(R,\rhd)
\rightarrowtail\Aut(R,\rhd)
\rightarrowtail\Sym(R).
\]
All these groups act on the rack~$R$, and we may write~$(g,x)\mapsto g\rhd x$ for these actions, as they extend the action of~$R$ on itself via the left-multiplications~$\ell_x$.


\subsection{Connected and homogeneous racks}

The canonical actions of the groups~$\Inn(R,\rhd)$ and~$\Aut(R,\rhd)$ on any rack~$(R,\rhd)$ lead to the following definitions.

\begin{definition}\label{def:orbits_connected}
Let~$R$ be a rack. The set of~\textit{orbits} of~$R$ under the action of~$\Inn(R,\rhd)$ via the left multiplications is denoted~$\Orb(R,\rhd)$. 
A rack~$R$ is called~\textit{connected} if~$\Orb(R,\rhd)$ consists of a single orbit.
This is equivalent to the action of the enveloping group~$\Gr(R,\rhd)$, or that of the free group~$\rmF(R)$, on~$R$ to have a single orbit. A rack is called~\textit{homogeneous} if the action of~$\Aut(R,\rhd)$ has a unique orbit.
\end{definition}

We note that the empty rack is not connected in this sense.

\begin{remark}
Connected racks are homogeneous, as~$\Inn(R,\rhd)$ is a subgroup of~$\Aut(R,\rhd)$, but the converse is false, as the following Example~\ref{ex:homogeneous_not_connected} shows.
\end{remark}

\begin{example}\label{ex:homogeneous_not_connected}
We give an example of a homogeneous quandle that is not connected. Consider the trivial quandle~$(R,\id_R)$ on a set~$R$. This means that we have~\hbox{$a\rhd b=b$} for all~$a,b\in R$ as in Example~\ref{ex:permutation_racks}. All permutations of the set~$R$ are automorphisms of this quandle, so that~$\Aut(R,\id_R)=\Sym(R)$, and this group acts transitively on~$R$. However, we have~$\Inn(R,\id_R)=\{\id_R\}$, which acts transitively if and only if~$R$ has at most one element.
\end{example}

If~$\alpha$ is an automorphism of a rack~$(R,\rhd)$, we have
\begin{equation}\label{eq:aut}
\ell_{\alpha(x)}=\alpha\ell_x\alpha^{-1}
\end{equation}
for all elements~$x$ in~$R$~(see~\cite[(1.5)]{Andruskiewitsch--Grana}).

Let~$(R,\rhd)$ be a connected rack and~$G=\Inn(R,\rhd)$. Any element~$x$ in~$R$ defines a 
surjection~$\phi_x\colon G\to R$ by evaluation, where the rack structure on~$G$ is given by~$\alpha\rhd\beta=\alpha\ell_x\alpha^{-1}\beta$. If~$H\leqslant G$ is the stabiliser of~$x$ in~$G$, which is the centraliser of the inner automorphism~$\mu=\ell_x$, then this surjection induces an isomorphism~\hbox{$R\cong G/H$} of racks, where the rack structure on the set of cosets is given by
\begin{equation}\label{eq:like-x-set}
\alpha H\rhd\beta H=\alpha\mu\alpha^{-1}\beta H.
\end{equation}
Similar arguments work in the homogeneous case, with the full automorphism group~$\Aut(R,\rhd)$ replacing the inner automorphism group~$\Inn(R,\rhd)$. The case of quandles is treated in~\cite[Thm.~7.1]{Joyce}. Conversely, if~$\mu$ is any element in a group~$G$ and~$H\leqslant\rmC_G(\mu)$, then~\eqref{eq:like-x-set} defines a rack structure on~$G/H$, and it is a quandle if and only if~$\mu\in H$.

\begin{remark}\label{rem:adapted}
The condition~$H\leqslant\rmC_G(\mu)$ is equivalent to~$[H,\mu]=\{e\}$, which means that the element~$\mu$ commutes with every element in the subgroup~$H$. This condition is stronger than necessary. It is necessary and sufficient that~$[H,\mu]$ is contained in all conjugates of~$H$. To see this, note that~$\alpha\mu\alpha^{-1}\beta H$ only depends on the coset of~$\beta$, but if we replace~$\alpha$ with~$\alpha h$ for some~$h$, we need to ensure~\hbox{$\alpha\mu\alpha^{-1}\beta H=(\alpha h)\mu(\alpha h)^{-1}\beta H$}, and this is equivalent to~\hbox{$\alpha[h,\mu]\alpha^{-1}\in H$}.
\end{remark}

There is a canonical~$G$--action on~$G/H$. This action is through automorphisms of the rack~$G/H$, as we have
\[
(g\alpha H)\rhd(g\beta H)=g(\alpha H\rhd\beta H),
\]
and the conjugates of~$\mu$ act as inner automorphisms. As the kernel of the action is the intersection of the conjugates of the subgroup~$H$, this allows us to determine when such a rack is connected.

\begin{remark}\label{rem:profile}
From~\eqref{eq:aut}, we see that if~$(R,\rhd)$ is a connected rack, or more generally a homogeneous rack, then all inner automorphisms are conjugate within the automorphism group~$\Aut(R,\rhd)$, and hence also in the permutation group~$\Sym(R)$. Therefore, as permutations, the inner automorphisms of a homogeneous rack all have the same cycle structure, and that cycle structure is an invariant of the homogeneous rack. This cycle structure is often called the~\textit{profile} of~$(R,\rhd)$. We will present a more conceptual approach to profiles in Section~\ref{sec:profiles}~(see Proposition~\ref{prop:profile} in particular).
\end{remark}


\section{Decompositions}\label{sec:decompositions}

The orbits of a rack~(Definition~\ref{def:orbits_connected}) are subsets that are mapped to themselves by all left-multiplications. The following weaker notion is fundamental for any structure theory.

\begin{definition}\label{def:subrack}
A subset~$S\subseteq R$ of a rack is a~\textit{subrack} if~$\ell_s(S)=S$ for all~$s$ in~$S$. 
\end{definition}

\begin{remark}
Of course, we can define whatever we like, but some care must be exercised to avoid common pitfalls. For instance, a subset of a rack closed under the operation~$\rhd$ need not be a subrack; the left-multiplications on the subset will still be injective but need not be surjective~(see Kamada's~\cite{Kamada:2010} for counterexamples). Definiton~\ref{def:subrack} avoids this problem by insisting on equality. In any event, there is no difference for finite racks, which are our primary concern, as injectivity and surjectivity are equivalent for self-maps of finite sets. Similar care is required when defining congruences. This has been clarified by Burrows--Tuffley~\cite{Burrows--Tuffley}. Again, in the case of racks and quandles that are finite, there is nothing to worry about.
\end{remark}

Being a subrack is weaker than being an orbit, as we do not require that we have~$\ell_r(S)=S$ for all~$r$ in~$R$. Here is a name for the stronger notion.

\begin{definition}\label{def:ideal}
We say that a subrack~$S\subseteq R$ is an~\textit{ideal} if it is preserved by all of~$R$, that is, if~$\ell_r(S)=S$ for all~$r\in R$. 
\end{definition}

Note that the empty set is an ideal of any rack, as is~$R$ itself.

A subrack~$S$ is a rack in its own right, and as the left-multiplications~$\ell_s$ are bijective, it follows that also~\hbox{$\ell_s(R\setminus S)=R\setminus S$} for all~$s$ in~$S$.
However, it does~\textit{not} follow that the set-theoretic complement~$R\setminus S$ is a subrack as well. For a subset~$S$ to be an ideal is equivalent to~$S$ being a subrack such that the complement~$R\setminus S$ is a subrack, too.

\begin{example}\label{ex:simple}
Let~$\big(\{(1,2),(1,3),(2,3)\},\rhd\big)$ be the conjugation rack that consists of the transpositions inside~$\Sym(3)$, the symmetric group on~$\{1,2,3\}$.
That is, we have~$a\rhd b=aba^{-1}$ for any two elements~\hbox{$a,b\in\{(1,2),(1,3),(2,3)\}$}. 
This is a quandle, and so every singleton, such as~$\{(1,2)\}$, is a subrack. However, the computation~$(1,3)\rhd(2,3)=(1,2)$ shows that the complement is not a subrack, so~$\{(1,2)\}$ is not an ideal.
\end{example}

We refer to~\cite{HSW, Kiani--Saki, Saki--Kiani} for more information on lattices of subracks and sub-quandles.


\begin{definition}\label{def:indecomposable}
A rack~$R$ is called~\textit{indecomposable} if it has precisely two ideals, namely~$\emptyset$ and~$R$ itself. This means that~$R$ is non-empty, and if we write~$R$ as a union of two disjoint subracks~$S$ and~$T$, then at least one of them must be empty~(compare with~\cite[Def.~1.13]{Andruskiewitsch--Grana}). 
\end{definition}

\begin{remark}
Definition~\ref{def:indecomposable} suggests to call such racks~`simple' instead. However, simple racks were defined differently by Andruskiewitsch and Gra\~na~\cite[Def.~3.3]{Andruskiewitsch--Grana}, following Joyce's work~\cite{Joyce:simple} on simple quandles.
\end{remark}

\begin{proposition}\label{prop:connected_implies_indecomposable}
A rack is connected if and only if it is indecomposable.
\end{proposition}

\begin{proof}
Whenever we have a decomposition of a rack~$R$ into subracks~$S$ and~$T$, the action of~$\Inn(R,\rhd)$ on~$R$ preserves this decomposition. Conversely, any decomposition of a rack~$R$ into subracks is a decomposition as~$\Inn(R,\rhd)$--set.
\end{proof}

\begin{example}\label{ex:Sym(3)}
If~$G=\Sym(3)$ is the conjugation quandle of the symmetric group on the three-element set~$\{1,2,3\}$, then~$G$ decomposes as the disjoint union of the indecomposable subracks~$\{e\},\{(1,2,3)\},\{(1,3,2)\}$ and~$\{(1,2),(1,3),(2,3)\}$. This can be interpreted as a refinement of the decomposition of a group into its conjugacy classes: the class~$\{(1,2,3),(1,3,2)\}$ of the~$3$--cycles, which forms an ideal, splits into two subracks that are not ideals. Besides, if~$Q=\{(1,2),(1,3),(2,3)\}$ denotes the conjugacy class of involutions, then the product~$Q\times Q$~(see Remark~\ref{rem:sums}) is an indecomposable quandle, but it is also the disjoint union of three indecomposable quandles, each isomorphic to~$Q$. This illustrates the fact that decomposability is defined in terms of complementary pairs. Every quandle can be written as a disjoint union of singletons, but this is useless from the point of view of a structure theory.
\end{example}

\begin{definition}\label{def:decomposition}
If a rack~$R$ is decomposable, then a \textit{decomposition} of the rack~$R$ is any way to write~$R$ as a union of disjoint subracks~$S$ and~$T$ that are non-empty. 
\end{definition}

Recall in this context that if~$S$ is a subrack, and the complement~$R\setminus S$ is a subrack, then~$S$ is an ideal, and so is~$R\setminus S$. The following observation will come in handy later.

\begin{lemma}\label{lem:connected_images}
Let a rack~$R$ be decomposed into subracks~$S$ and~$T$. The image of every morphism~\hbox{$C\to R$} from an indecomposable rack~$C$ is entirely contained in~$S$ or in~$T$.
\end{lemma}

\begin{proof}
Assume that we have~$s$ and~$t$ in~$C$ which map to~$S$ and~$T$, respectively. As~$C$ is connected, there is an element~$c$ in the group~$\Inn(C,\rhd)$ such that~\hbox{$c\rhd s=t$}. But then the image of~$c$ in~$\Inn(R,\rhd)$ maps the image of~$s$ to the image of~$t$, contradicting the fact that~$\Inn(R,\rhd)$ preserves the decomposition of~$R$ into~$S$ and~$T$.
\end{proof}

\begin{remark}
While every~$G$--map between homogeneous~$G$--spaces is surjective, not every morphism between connected racks is surjective. For instance, if~$Q$ is a connected quandle, any constant map~$Q\to Q$ is a morphism, but clearly not surjective in general.
\end{remark}


\subsection{Disjoint unions and irreducible racks}

Brieskorn~\cite[\S2]{Brieskorn} has noted that there is a canonical rack structure on the disjoint union of any two racks: 

\begin{proposition}\label{prop:disjoint_union_of_racks}
If~$R$ and~$S$ are two racks, then there is a rack structure on the disjoint union~\hbox{$R\sqcup S$} with
\[
a\rhd b=
\begin{cases}
a\rhd b & a,b\in R\text{ or }a,b\in S\\
b & \text{otherwise}.
\end{cases}
\]
The disjoint union defines a symmetric monoidal structure on the category of all racks. It preserves the full subcategory of finite racks.
\end{proposition}
 
The proof is a trivial verification.

\begin{remark}
Any idempotent endomorphism~$e\colon R\to R$ of a rack~$R$ splits: the image~$e(R)$ is a subrack of~$R$ with the induced structure, and~$e$ factors as a composition~$R\to e(R)\to R$, where the composition in the other order is the identity on~$e(R)$. However, note that~$e(R)$ need not be an ideal in~$R$, and therefore, the rack~$R$ need not be decomposed into~$e(R)$ and its complement~\hbox{$R\setminus e(R)$}. Besides, there are racks~$R$ that are not retracts of the rack~$R\sqcup\star$. This happens, for instance, if there is no rack morphism from~$\star$ to~$R$, unlike for quandles.
\end{remark}

\begin{remark}\label{rem:indec_=>_irred}
The disjoint union~$R=S\sqcup T$ as in Proposition~\ref{prop:disjoint_union_of_racks} is decomposable into the two subracks~$S$ and~$T$, but not every decomposition arises in this way, because for a decomposition, we do not require that~$S$ and~$T$ act trivially~(i.e., as pointwise stabilisers) on each other. 
\end{remark}

Brieskorn also sketched how to write any given rack uniquely as a disjoint union of subracks. We can understand this by noting that a partition of a set~$R$ as a disjoint union of non-empty subsets~$R_j$ is the same thing as an equivalence relation on~$R$ with the~$R_j$ as the equivalence classes. If~$R$ is a rack, we are interested in equivalence relations~$\sim$ where inequivalent elements act trivially on each other,~i.e., where we have
\[
a\rhd b=b\text{ and }b\rhd a=a\text{ if }a\not\sim b.
\]
Therefore, if we define
\begin{equation}\label{eq:eqrel}
a\sim b\text{ if }a\rhd b\not=b\text{ or }b\rhd a\not=a,
\end{equation}
this generates the finest equivalence relation on~$R$ that partitions~$R$ into subracks~$R_j$ with the property that all elements of~$R_j$ fix~$R_k$ if~$j\not=k$. 

\begin{definition}
Following Brieskorn, we call a rack~\textit{irreducible} if this partition is trivial, with a unique equivalence class of elements. This is the case if and only if~$R\not=\emptyset$ and~$R=S\sqcup T$ implies~$S=\emptyset$ or~$T=\emptyset$.
\end{definition}

\begin{remark}
There is a graphical way to think about the canonical partition of a rack into a disjoint union of subracks: every rack~$R$ defines a graph with vertices the elements of~$R$ and an edge between~$a$ and~$b$ if~$a$ and~$b$ `commute' in the sense that~$a\rhd b=b$ and~$b\rhd a=a$, similarly to the definition of the graph associated with a right-angled Artin group. The connected components of the complementary~(!) graph are the equivalence classes that partition~$R$ as a disjoint union~$\sqcup_jR_j$ of subracks.
\end{remark}

\begin{proposition}\label{prop:indec_=>_irred}
Indecomposable racks are irreducible.
\end{proposition}

\begin{proof}
Assume that the rack~$(R,\rhd)$ is reducible, i.e., we can write~$R$ as a disjoint union~\hbox{$R=S\sqcup T$} with~\hbox{$S\not=\emptyset\not=T$}. Then, no inner automorphism of~$R$ moves an element from~$T$ into~$S$, as all its left-multiplications map~$T$ into~$T$. This shows that~$R$ is not connected. By contrapositive, indecomposable racks must be irreducible~(see also Remark~\ref{rem:indec_=>_irred} above).
\end{proof}
 
\begin{example}
Here is an example of a rack of order three that is neither irreducible nor homogeneous: take the disjoint union~$R=S\sqcup T$ of the permutation rack~$S=(\{1,2\},(1,2))$ and the singleton rack~$T=\star$. By construction, the rack~$R$ is not irreducible. It is also not homogeneous, as no automorphism of~$R$ maps~$T$ into~$S$.
\end{example}

\begin{example}\label{ex:(12)(34)}
Here is an example of a rack that is irreducible and homogeneous but not indecomposable or connected: take the permutation rack~$(R,\pi)$ on the set~\hbox{$R=\{1,2,3,4\}$} with~\hbox{$\pi=(12)(34)$}. The inner automorphism group is generated by~$\pi$, so that this rack is not connected; it has the obvious decomposition into the cycles of~$\pi$. However, the full automorphism group is dihedral of order 8, so this rack is homogeneous~(compare~\cite[Ex.~1.4]{Andruskiewitsch--Grana}). Also, the left multiplication by any element of~$R$ has no fixed point on~$R$, showing that~$R$ is irreducible. 
More generally, note that any permutation rack~$(R,\pi)$ such that~$\pi$ does not fix any letter of~$R$ is irreducible (since~$a\rhd b=\pi(b)\neq b$ for all~$a,b\in R$). But such rack is connected if and only if~$\pi$ is a cycle of length~$|R|$.
\end{example}

\begin{example}\label{ex:irred_disconn_inhomog}
To have an example of a rack that is irreducible but not homogeneous, take the permutation rack~$(R,\pi)$ with~$R=\{1,2,3\}$ and~$\pi=(1,2)$. Then~$R$ is irreducible, since we have~\hbox{$r\rhd 1=2\neq 1$} for all elements~$r\in R$, and so every element of~$R$ is equivalent to~$1$. But~$(R,\pi)$ is not a homogeneous rack, as every automorphism of~$(R,\pi)$ must fix~$3$.
\end{example}

\begin{example}\label{ex:red_and_hom}
Here is an example of a rack that is homogeneous but not irreducible. Take~$\pi=\id_R$ on any set~$R$ that has at least two elements. Then~$(R,\id_R)$ is a homogeneous rack, since~$\Aut(R,\id_R)=\Sym(R)$ acts transitively on~$R$. On the other hand, as~$R$ is the disjoint union of singletons, it is not irreducible.
\end{example}

We will say more on permutation racks in Section~\ref{sec:perm}.


The following diagram shows the implications that hold between some of the properties of racks discussed so far.
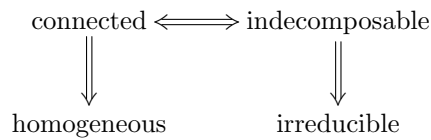
\[
\xymatrix{
\text{connected}\ar@{<=>}[r]\ar@{=>}[d] & \text{indecomposable}\ar@{=>}[d]\\
\text{homogeneous} & \text{irreducible}
}
\]
We have seen examples illustrating that no other implications hold.


Given a rack~$(R,\rhd)$, it seems tempting to try and use the orbits of the~$\Inn(R,\rhd)$--action to find a canonical decomposition of a rack into indecomposable racks. However, this can only be the first step: while the group~$\Inn(R,\rhd)$ acts transitively on each orbit~$S\subseteq R$, the group~$\Inn(S,\rhd)$ no longer needs to, and we will have to study its orbits, and then continue, to get a tree-like\label{tree-like} structure of further refinements. Here, in addition to Example~\ref{ex:Sym(3)}, is another specimen that illustrates this phenomenon.

\begin{example}\label{ex:tree}
Orbits can be decomposable. Let~$\Dih(n)$ be the dihedral quandle of order~$n$, i.e., we take~$\Dih(n)=\bbZ/n$ as a set and~$a\rhd b=2a-b$. The name comes from the fact that this is isomorphic to the conjugacy class of the reflections in the dihedral group of order~$n$. For instance, in~$\Dih(4)$, the classes of~$0$ and~$2$ interchange~$1$ and~$3$, and conversely. Thus, the quandle~$\Dih(4)$ has two orbits. Both are isomorphic to the trivial quandle~$\Dih(2)$ with two elements. More generally, in the quandle~$\Dih(2^k)$, every element is an involution with exactly two fixed points. This quandle decomposes into two quandles isomorphic to~$\Dih(2^{k-1})$. We see the process to decompose~$\Dih(2^k)$ completely into~$2^k$ singletons needs~$k$ steps. This is the \textit{depth} of~$\Dih(2^k)$ as defined by Nelson and Wong~\cite{Nelson--Wong}.
\end{example}


\section{Additive invariants and the Burnside ring of racks}\label{sec:Burnside}

In this section, we introduce the main objects of study in this paper, the Burnside rings of finite racks and finite quandles. To motivate the construction, we first look at Burnside rings of finite groups. Later, we will also consider their crossed variants~(see Section~\ref{sec:crossed}).

Let~$G$ be a finite group. The Grothendieck ring~$\rmK_0(\calS^G,\sqcup,\star)$ of the full category of the category~$\calS^G$ consisting of finite~$G$--sets with respect to disjoint union is, by definition, the~\textit{Burnside ring}~$\rmB(G)$ of~$G$. We refer to~\cite{Solomon, Dress} for original sources. The basic facts about these rings can be found, for instance, in~\cite{tomDieck:1,tomDieck:2,Bouc:Burnside}. For example, it is well known that if~$G$ is a finite group, then an integral basis for the Burnside ring of~$G$ is the set of isomorphism classes of transitive~$G$--sets~$[G/H]$, where~$H$ is taken from the~$G$--orbits of subgroups under conjugation: there is an isomorphism~$G/H\cong G/K$ of~$G$--sets if and only if~$H$ and~$K$ are~$G$--conjugate. 

Before we can define the Burnside rings, we observe that, unlike for finite~$G$--sets, we cannot use categorical sums as in Remark~\ref{rem:sums} or disjoint unions as in Proposition~\ref{prop:disjoint_union_of_racks}, as these are not suitable for our purposes~(see also Section~\ref{sec:Grothendieck} below). Instead, we proceed more carefully and replace these constructions by deconstructions, or decompositions. 


\subsection{Additive invariants}

For a conceptual approach, it will be beneficial to use the following terminology, which is adapted to our context from~\cite[IV.1]{tomDieck:2}. 

\begin{definition}
An~\textit{additive invariant} of racks is an abelian group~$B$ together with a family~$b$ of elements~$b(R)\in B$, one for each rack~$R$, such that we have~$b(R_1)=b(R_2)$ whenever there exists an isomorphism~$R_1\cong R_2$ of racks and~\hbox{$b(R)=b(S)+b(T)$} if the rack~$R$ decomposes into subracks~$S$ and~$T$.
\end{definition}

\begin{example}\label{ex:cardinality}
The abelian group~$\bbZ$ together with the elements~$|R\,|\in\bbZ$ equal to the cardinality of~$R$ forms an additive invariant.
\end{example}

We will see many less-trivial examples later. For now, let us point out some non-examples:

\begin{example}
The number of orbits is~\textit{not} an additive invariant, as Example~\ref{ex:tree} shows. It features a rack~$R$ that decomposes into two orbits~$S$ and~$T$, but~$S$ and~$T$, considered as racks in their own right, have two orbits each.
\end{example}

\begin{example}
Kai and Tamaru~\cite{Kai--Tamaru} have defined an Euler characteristic for quandles, but it is~\textit{not} an additive invariant either, as their examples show.
\end{example}

\begin{example}\label{ex:tau}
If~$R$ is any rack, we can consider the subset
\[
\tau(R) = \{y\in R\mid x\rhd y = y\text{ for all }x\in R\}.
\]
This is the see of fixed points of the canonical~$\Inn(R,\rhd)$--action on~$R$. It is easy to check that~$\tau(R)$ is an ideal, so that we have a decomposition of~$R$ into~$\tau(R)$ and its complement~$R\setminus\tau(R)$, and~$\tau(R)$ is a trivial quandle. We can see that
\[
R\longmapsto|\tau(R)|
\]
does not define an additive invariant of racks as follows. If~$R$ is connected, there are two cases. On the one hand, we can have equality~\hbox{$\tau(R)=R$}. Then~$\tau(R)$ is connected, and so~$R=\tau(R)=\star$. On the other hand, we can have~$\tau(R)=\emptyset$. Then~$R$ cannot be a singleton. If we look at the disconnected Example~\ref{ex:tree}, we have a rack~$R$ of order~$4$ with~$\tau(R)=\emptyset$. However, there is a decomposition~\hbox{$R=S\cup T$} with~$S$ and~$T$ both trivial of order~$2$, so~$\tau(S)=S$ and~$\tau(T)=T$. This non-example will show up again in Remark~\ref{rem:lambda}.
\end{example}

\begin{example}\label{ex:tau'}
A slight variant of the preceding Example~\ref{ex:tau} is given by
\[
\tau'(R) = \{x\in R\mid 
x\rhd x = x\text{ and }
x\rhd y = y\text{ for all } y\in R\text{ with }y\rhd y=y\}.
\]
In particular, if~$Q$ is a quandle, then we have
\[
\tau'(Q) = \{x\in Q\mid x\rhd y = y\text{ for all } y\}.
\]
Generally, it is straightforward to check that~$\tau'(R)$ is a subrack, but this time it is unclear whether it is an ideal. Either way, for the rack~$R=S\cup T$ as in the previous example, we have~$\tau'(R)=\emptyset$, but~$\tau'(S)\not=\emptyset\not=\tau'(T)$, so that~$\tau'$ does not define an additive invariant either.
\end{example}


\subsection{Universal additive invariants}

If~$(B,b)$ is an additive invariant, and~$C$ is another abelian group, and~$\phi\colon B\to C$ is a homomorphism, then~$C$, together with the family~\hbox{$c=\phi(b)$} of elements~$c(R)=\phi(b(R))$ is also an additive invariant. 

\begin{definition}
An additive invariant~$(A,a)$ is~\textit{universal} if  every other additive invariant~$(B,b)$ has the form~$b=\phi(a)$ for a unique homomorphism~$\phi\colon A\to B$. 
\end{definition}

A universal additive invariant is determined up to a distinguished isomorphism. This means that for any other universal additive invariant~$(A',a')$, there exists a unique isomorphism~$A\cong A'$ under which~$a$ and~$a'$ correspond to each other.

\begin{definition}\label{def:B(R)}
We define the abelian group~$\rmB(\calR)$ in terms of generators and relations. It is generated by classes~$b(R)$ of finite racks~$R$ modulo the relations~$b(R_1)=b(R_2)$ if~$R_1\cong R_2$ and~\hbox{$b(R)=b(S)+b(T)$} whenever we have a decomposition of a rack~$R$ into subracks~$S$ and~$T$. We already call~$\rmB(\calR)$ the~\textit{Burnside ring} of finite racks, even though we will define the product on it only later.
\end{definition}

Note that we always have a trivial decomposition of~$R$ into~$R$ and the empty rack~$\emptyset$. It follows that~$b(\emptyset)=0$ in~$\rmB(\calR)$.

\begin{lemma}
Every element in the Burnside ring~$\rmB(\calR)$ of finite racks can be written as a difference~\hbox{$b(R_+)-b(R_-)$} where the~$R_\pm$ are finite racks.
\end{lemma}

\begin{proof}
Every~$\bbZ$--linear combination is a difference of~$\bbZ$--linear combinations with positive coefficients. So it suffices to show that every~$\bbZ$--linear combination with positive coefficients can be represented by a rack. But that is clear by induction, since~\hbox{$2b(R)=b(R\sqcup R)$} and~\hbox{$b(R)+b(S)=b(R\sqcup S)$}.
\end{proof}

\begin{proposition}\label{prop:universal_additive_invariant}
The abelian group~$\rmB(\calR)$, together with the family of elements~$b(R)$, is a universal additive invariant for racks. 
\end{proposition}

\begin{proof}
This is entirely formal. If~$(C,c)$ is an additive invariant, we can define~$\phi\colon\rmB(\calR)\to C$ by sending~$b(R)$ to~$c(R)$. By definition, this is a homomorphism of abelian groups, and it is the unique homomorphism that satisfies~$c(R)=\phi(b(R))$ for all racks~$R$.
\end{proof}


\subsection{Relation to the Grothendieck construction}\label{sec:Grothendieck}

We need to contrast the definition of the Burnside ring~$\rmB(\calR)$ of racks as the universal additive invariant with the Grothendieck construction~$\rmK_0(\calR,\sqcup,\star)$ for the full subcategory of~$\calR$ consisting of the finite racks with respect to the symmetric monoidal structure~$\sqcup$ from Proposition~\ref{prop:disjoint_union_of_racks}. The abelian group~$\rmK_0(\calR,\sqcup,\star)$ is the group completion of the abelian monoid of isomorphism classes~$[R]$ of finite racks. It is generated by classes~$[R]$ represented by racks~$R$, and we have the relations~\hbox{$[S\sqcup T]=[S]+[T]$}. Every element can be written as a difference~$[R]-[S]$, and we have~$[R]=[S]$ if and only if~$R\sqcup T\cong S\sqcup T$ for some~$T$.

There is a surjective homomorphism
\begin{equation}\label{eq:K0_to_BR_surjection}
\rmK_0(\calR,\sqcup,\star)\longrightarrow\rmB(\calR)
\end{equation}
that is the identity on the representatives, i.e., such that~$[R]\mapsto b(R)$ for any rack~$R$.
To see this, note that a disjoint union of racks~\hbox{$R=S\sqcup T$} yields a decomposition of~$R$, and therefore, we have~\hbox{$b(R)=b(S)+b(T)$} on the right-hand side of~\eqref{eq:K0_to_BR_surjection}.

Because of the difference between irreducibility and indecomposability, the homomorphism~\eqref{eq:K0_to_BR_surjection} is not injective. In general, the Grothendieck construction~$\rmK_0(\calR,\sqcup,\star)$ is too big to be useful. 

\begin{remark}
The relation between the K-theory group and the Burnside ring can be compared with the relation between the Euler ring~$\rmU(G)$ and the Burnside ring~$\rmB(G)$ of a compact Lie group~$G$~(see~\cite[Sec.~5.4]{tomDieck:1} and~\cite[Sec.~IV.1]{tomDieck:2} for definitions). There is always a surjection~$\rmU(G)\to\rmB(G)$, and it is an isomorphism for finite groups~$G$, but not in general.
\end{remark}

Two racks~$R$ and~$S$ with~\hbox{$b(R)=b(S)\in\rmB(\calR)$} need not be isomorphic as racks. 

\begin{example}\label{ex:not_injective}
Let us revisit Example~\ref{ex:(12)(34)} of a rack~$R$ that is irreducible but decomposable into~$S$ and~$T$, say. Then, we have
\[
b(R)=b(S)+b(T)=b(S\sqcup T)
\]
in~$\rmB(\calR)$, even though the racks~$R$ and~$S\sqcup T$ are not isomorphic. Therefore, the difference~\hbox{$[R]-[S\sqcup T]$} is an element in~$\rmK_0(\calR,\sqcup,\star)$ that is in the kernel of~\eqref{eq:K0_to_BR_surjection}. This element is non-zero. To see this, recall from Section~\ref{sec:decompositions} that there is an essentially unique way to write any finite rack~$X$ as a disjoint union~\hbox{$X=X_1\sqcup\dots\sqcup X_m$} of irreducible subracks~$X_j$. The function~\hbox{$X\mapsto m$} is additive with respect to disjoint union and defines a homomorphism~$\rmK_0(\calR,\sqcup,\star)\to\bbZ$ of abelian groups that sends~$[R]$ to~$1$ and~$[S\sqcup T]=[S]+[T]$ to~$2$, showing that~$[R]\not=[S\sqcup T]$ in~$\rmK_0(\calR,\sqcup,\star)$. 
\end{example}

\begin{remark}
In contrast to the previous Example~\ref{ex:not_injective}, we will see in Theorem~\ref{thm:inj_on_connected} below that a difference~\hbox{$[C]-[D]$} of \textit{connected} racks is in the kernel of~\eqref{eq:K0_to_BR_surjection} only if~$C\cong D$, so that~$[C]=[D]$. On the other hand, the map from the set of isomorphism classes of finite racks to~$\rmK_0(\calR,\sqcup,\star)$ is injective if and only if~$R\sqcup T\cong S\sqcup T$ implies~$R\cong S$ for all~$S$ and~$T$. Corollary~\ref{cor:dec_cor} is a result in this direction.
\end{remark}

\begin{remark}
The Grothendieck group~$\rmK_0(\calR,\sqcup,\star)$ is the group of components of an algebraic~K-theory space~(or spectrum). It turns out that the Burnside ring~$\rmB(\calR)$ is also the group of components of a~K-theory space~(or spectrum), albeit~\textit{not} one based on the symmetric monoidal structure given by~$\sqcup$. The details will appear elsewhere. It is worth, however, to point out that this is \textit{not} the K-theory of the Lawvere theory of racks in the sense of~\cite{Bohmann--Szymik:1}, which is formed using the symmetric monoidal category of finitely generated free racks with respect to the sum~(categorical coproduct); as we have already mentioned several times, these coproducts are far from finite except for trivial cases.
\end{remark}


\subsection{Products}

From their definition, the Burnside rings are only abelian groups. Now we will establish their multiplicative structure as commutative rings. We achieve this goal with the following result. It uses the product of racks as explained in Remark~\ref{rem:sums}.

\begin{proposition}\label{prop:AR_multiplication}
The abelian group~$\rmB(\calR)$ admits the structure of a commutative ring with unit with respect to a product that satisfies
\begin{equation}\label{eq:AR_multiplication}
b(R')b(R)=b(R'\times R)
\end{equation}
and with multiplicative identity the class~$b(\star)$ of the singleton rack~$\star$.
\end{proposition}

We start with the following evident observation.

\begin{lemma}\label{lem:product_decompositions}
If a rack~$R$ is decomposed into~$S$ and~$T$, then~$R'\times R$ is decomposed into~$R'\times S$ and~$R'\times T$, and similarly with the roles of the factors interchanged.
\end{lemma}

\begin{proof}[Proof of Proposition~\ref{prop:AR_multiplication}]
The preceding lemma shows that the product~\eqref{eq:AR_multiplication} is well-defined. The associativity, commutativity, and the neutrality of~$\star$ are then obvious. 
\end{proof}

\begin{definition}
The commutative ring~$\rmB(\calR)$ in Proposition~\ref{prop:AR_multiplication} is the~\textit{Burnside ring of finite racks}.
\end{definition}


\subsection{The Burnside ring of quandles}\label{sec:B(Q)}

The construction we have used for the Burnside ring~$\rmB(\calR)$ of finite racks works for quandles, too, and it results in the~\textit{Burnside ring~$\rmB(\calQ)$ of finite quandles}. In fact, as quandles are special kinds of racks, we have one way of making the relationship between their Burnside rings more precise:

\begin{proposition}\label{prop:quandles_split}
There are morphisms
\[
\rmB(\calQ)\longrightarrow\rmB(\calR)\longrightarrow\rmB(\calQ)
\]
of rings whose composition is the identity. 
\end{proposition}

This result allows us to think of~$\rmB(\calQ)$ as a subring of~$\rmB(\calR)$.

\begin{proof}
The first morphism is just given by the identity on representatives, i.e., by considering a quandle as a rack.

The second morphism is given by associating to a rack~$(R,\rhd)$ with canonical automorphism~$\sigma$ given on elements by~\hbox{$\sigma(x)=x\rhd x$} the quandle that is supported on the same set~$R$ with the binary operation~\hbox{$x\rhd' y=\sigma^{-1}(x\rhd y)$}, as in~\cite[Prop.~4.3]{Szymik:2018}. This construction is compatible with decompositions and products. Therefore, it induces a morphism of rings, as claimed. 

For quandles, we have~\hbox{$\sigma=\id$}, by definition, and this shows that the composition is the identity.
\end{proof}


\subsection{Functorial additive invariants}

The following remark fits best here, but we will only refer to it in Section~\ref{sec:crossed}.

\begin{remark}\label{rem:functorial}
We can easily extend the definitions to define~\textit{functorial} additive invariants of racks along the lines of~\cite[Sec.~2]{Luck}. Such an invariant consists of a functor~$B$ from the category of finite racks to the category of abelian groups, and for each finite rack~$R$ an element~$b(R)\in B(R)$ such that for any isomorphism~$f\colon R\to S$ the homomorphism~\hbox{$B(f)\colon B(R)\to B(S)$} maps~$b(R)$ to~$b(S)$ and if~$R$ is the disjoint union of two subracks~$S$ and~$T$, then~\hbox{$b(R)=B(j_S)b(S)+B(j_T)b(T)$}, where~$j_S$ and~$j_T$ are the inclusions~$S\to R$ and~$T\to R$, respectively. For example, an additive invariant is a functorial additive invariant with a constant functor.
If~$B$ is a functorial additive invariant and~$\Phi\colon B\to C$ is a natural transformation into another functor~$C$ from the category of finite racks to the category of abelian groups, then the family formed by all~$c(R)=\Phi(R)b(R)$ defines another functorial additive invariant. For example, if we choose~$C$ to be the constant functor with value~$B(\star)$, then there is a canonical natural transformation~$B\to B(\star)$, and this defines an~(ordinary) additive invariant from any functorial additive invariant.
A functorial additive invariant~$(A,a)$ is~\textit{universal} if for any other functorial additive invariant~$(B,b)$ there is a unique natural transformation~$\Phi\colon A\to B$ such that~$b=\Phi(a)$ as above. If~$(A,a)$ is a universal functorial additive invariant, then the images in~$A(\star)$ define a(n ordinary) universal additive invariant.
To construct the universal functorial additive invariant of finite racks, we define for any finite rack~$R$ the abelian group~$\rmB({\calR\!\downarrow\!R})$ generated by classes~\hbox{$b(X\to R)$} of finite racks over~$R$ and subject to the relations~$b(X\to R)=b(Y\to R)$ whenever~$X\to R$ and~$Y\to R$ are isomorphic over~$R$, that is, there exists an isomorphism~$X\to Y$ such that the diagram
\[
\xymatrix{X\ar[rr]\ar[dr]&&Y\ar[dl]\\&R}
\]
commutes, and
$b(X\to R)=b(Y\to R)+b(Z\to R)$ whenever~$X$ is decomposed into subracks~$Y$ and~$Z$, and the maps are the restrictions. This defines, by composition, a functor~$R\mapsto\rmB({\calR\!\downarrow\!R})$ from the category of finite racks to the category of abelian groups, and the classes~\hbox{$b(R=R)\in\rmB({\calR\!\downarrow\!R})$} of the indentities give the universal functorial additive invariant.
\end{remark}


\section{An integral basis for the Burnside ring of racks}\label{sec:basis}

In this section, we describe an integral basis for the Burnside ring of racks.
Let~$\Rcon$ be the set of isomorphism classes of finite racks that are connected. 
Here is the main result.

\begin{theorem}\label{thm:basis}
The elements~$b(R)$ represented by the connected racks~$R$ form an integral basis of the abelian group~$\rmB(\calR)$. In other words, the homomorphism
\begin{equation}\label{eq:basis}
\bbZ\{\Rcon\}\longrightarrow\rmB(\calR), [R]\longmapsto b(R),
\end{equation}
from the free abelian group~$\bbZ\{\Rcon\}$ with basis~$\Rcon$ to~$\rmB(\calR)$ is an isomorphism of abelian groups. In particular, the abelian group~$\rmB(\calR)$ is torsion-free.
\end{theorem}

\begin{proof}
We begin by proving the surjectivity of the homomorphism. The abelian group~$\rmB(\calR)$ is defined to be generated by the isomorphism classes of all finite racks. Therefore, it suffices to show that the class of any finite rack can be written as a linear combination of classes of connected finite racks. 
We can see this by induction on the number of the elements of~$R$. The empty rack is not connected but represents the zero element in the group~$\rmB(\calR)$, so we need not worry about it. If a rack~$R$ has exactly one element, it is connected, and we are done. 
Suppose now that~$R$ has more than one element and that the statement is true for all racks with fewer elements than~$R$. If~$R$ is connected, we are immediately done. If not, we can decompose~$R$ into subracks~$S$ and~$T$ with fewer elements. By induction hypothesis, we can write both~$b(S)$ and~$b(T)$ as linear combinations of~$b$'s of connected racks. Therefore, so can~$b(R)=b(S)+b(T)$.

To prove injectivity, we construct a homomorphism
\begin{equation}\label{eq:basis_inverse}
\rmB(\calR)\longrightarrow\bbZ\{\Rcon\}
\end{equation}
whose composition with~\eqref{eq:basis} is the identity. It then follows that~\eqref{eq:basis} is injective as well, and therefore an isomorphism with inverse~\eqref{eq:basis_inverse}.

To construct the homomorphism~\eqref{eq:basis_inverse}, we recall that this is the same as an additive invariant of racks with values in~$\bbZ\{\Rcon\}$. We will define this invariant as follows. Take any finite rack~$R$. It can be written as the disjoint union of its maximal connected subracks~(see~\cite[Prop.~1.17]{Andruskiewitsch--Grana}). More formally, if~$\Pi(R)$ is the set of maximal connected subracks of~$R$, then 
\[
R=\bigcup_{C\in\Pi(R)}\!\!\!C.
\] 
We claim that
\begin{equation}\label{eq:sum}
R\longmapsto\sum_{C\in\Pi(R)}\!\!\![C]
\end{equation}
is an additive invariant of finite racks with values in the abelian group~$\bbZ\{\Rcon\}$. First, it is obvious that the sum in~\eqref{eq:sum} only depends on the isomorphism class of~$R$. This shows invariance under isomorphisms. Second, whenever we have a decomposition of~$R$ into~$S$ and~$T$ as in Definition~\ref{def:decomposition}, we have to show that
\[
\sum_{C\in\Pi(R)}\!\!\![C]=\sum_{C\in\Pi(S)}\!\!\![C]+\sum_{C\in\Pi(T)}\!\!\![C].
\]
This property follows from Lemma~\ref{lem:connected_images}: every~$C$ in~$\Pi(R)$ either is contained in~$S$ or in~$T$, and therefore, the set~$\Pi(R)$ is the disjoint union of~$\Pi(S)$ and~$\Pi(T)$. This shows additivity. Both properties establish that we have an additive invariant, and thus, a homomorphism~\eqref{eq:basis_inverse} with~$b(R)\mapsto\sum_C[C]$.

It remains to be seen that the composition with~\eqref{eq:basis} is the identity. But this is straightforward from the definitions. Let~$D$ be a finite rack that is connected, so that~\hbox{$\Pi(D)=\{D\}$}. Then, the composition is
\[
[D]\longmapsto b(D)\longmapsto\sum_{C\in\Pi(D)}\!\!\![C]=\sum_{C\in\{D\}}\!\!\![C]=[D],
\]
which clearly is the identity.
\end{proof}


\begin{theorem}\label{thm:inj_on_connected}
If~$C$ and~$D$ are connected racks such that~$b(C)=b(D)$ in~$\rmB(\calR)$, then~$C\cong D$.
\end{theorem}

\begin{proof}
This follows immediately from Theorem~\ref{thm:basis}.
\end{proof}

\begin{corollary}\label{cor:dec_cor}
Let~$R$ be a finite rack that has two decompositions~\hbox{$R= S_1\cup T_1$} and~$R=S_2\cup T_2$. If there is an isomorphism~$T_1\cong T_2$ and the subracks~$S_j$ are connected, then there is also an isomorphism~\hbox{$S_1\cong S_2$}.
\end{corollary}

\begin{proof}
We have~$b(T_1)=b(T_2)$ and
\[
b(S_1)+b(T_1)=b(S_1\cup T_1)=b(R)=b(S_2\cup T_2)=b(S_2)+b(T_2).
\]
These imply~$b(S_1)=b(S_2)$. When the~$S_j$ are connected, we can use Theorem~\ref{thm:inj_on_connected} to deduce that they are isomorphic.
\end{proof}

It would be interesting to decide whether the connectivity of the~$S_j$ is necessary for this result.


\section{Marks}\label{sec:marks}

Recall that, in the study of Burnside rings of finite groups, each orbit~$G/H$ leads to a ring homomorphism~$\Phi_H\colon\rmB(G)\to\bbZ$ that sends the class of a finite~$G$--set~$X$ to the number of~$H$--fixed points in~$X$. As~$X^H\cong\Mor_{\calS^G}(G/H,X)$, this suggests the following analogue for racks.

\begin{proposition}\label{prop:marks}
Let~$C$ be a connected rack that is finitely generated but not necessarily finite. The assignment
\[
R\longmapsto|\Mor_\calR(C,R)|
\]
defines an additive invariant of finite racks. The resulting map
\begin{equation}\label{eq:C_mark}
\Phi_C\colon\rmB(\calR)\longrightarrow\bbZ
\end{equation}
is a homomorphism of commutative rings.
\end{proposition}

\begin{definition}
The morphism~$\Phi_C$ in~\eqref{eq:C_mark} is the~\textit{mark} defined by the rack~$C$.
\end{definition}

\begin{proof}[Proof of Proposition~\ref{prop:marks}]
As~$C$ is finitely generated, there are only finitely many rack morphisms from~$C$ into any finite rack~$R$. This allows us to count them. By Lemma~\ref{lem:connected_images}, the assignment is additive. Therefore, it induces a homomorphism~$\Phi_C$ of abelian groups. As the functor given by~\hbox{$R\mapsto\Mor_\calR(C,R)$} is, more obviously, also compatible with products, 
\[
\Mor_\calR(C,R\times S)\cong \Mor_\calR(C,R)\times \Mor_\calR(C,S),
\]
the homomorphism~\eqref{eq:C_mark} is also multiplicative, i.e., it is a morphism of commutative rings.
\end{proof} 

\begin{example}
Let us consider~$C=\FR(1)$, the free rack on one generator as in Example~\ref{ex:free}. This is clearly finitely generated, and it is also connected~(even homogeneous). By definition, we have a natural bijection~\hbox{$\Mor_\calR(C,R)\cong R$} of sets, and the corresponding homomorphism~\hbox{$\rmB(\calR)\to\bbZ$} sends the class represented by~$R$ to the number of elements in~$R$. This recovers Example~\ref{ex:cardinality}.
\end{example}

\begin{example}\label{ex:another_retraction}
For~$C=\star$, we have
\[
\Mor_{\calR}(\star,R)\cong\{q\in R\mid q\rhd q =q\}.
\]
As this is the fixed set of the canonical automorphism~$\sigma$, let us write~$R^\sigma$ for it. This set can also be thought of as the maximal sub{\it-quandle} of~$R$. It is even an ideal: for any such~$q$ and an arbitrary~$r$ in~$R$, we have
\[
r\rhd q=r\rhd(q\rhd q)=(r\rhd q)\rhd(r\rhd q),
\]
so~$r\rhd q\in R^\sigma$, too. Therefore, we have a decomposition of~$R$ into~$R^\sigma$ and its complement. The additive invariant~$\Phi_\star(R)=|R^\sigma|$ merely counts the number of elements, but of course we can factor~$\Phi_\star$ through~$\rmB(\calQ)$ by taking the quandle structure into account.
\[
\xymatrix{
&\rmB(\calQ)\ar[rd]&\\
\rmB(\calR)\ar[ru]\ar[rr]&&\bbZ.
}
\]
The arrow~$\rmB(\calR)\to\rmB(\calQ)$ thus constructed is different from the~`untwisting' homomorphism described in Proposition~\ref{prop:quandles_split} above, which preserves the underlying set. Both agree on~$\rmB(\calQ)$, though.
\end{example}

\begin{example}
If, more generally, the rack~$C=C_n$ under consideration is chosen to be an~$n$--cycle for some~\hbox{$n\geqslant 1$}, then~\hbox{$|\Mor_\calR(C_n,R)|$} counts the number of tuples~$(r_1,\dots,r_n)\in R^n$ such that~\hbox{$r_j\rhd r_k = r_{r+1}$} for all~$j$ and~$k$, with indices read mod~$n$. If~$R$ is a quandle, this is just the diagonal~$r_1=\dots=r_n$, which has the same number of elements as~$R$, but for racks that are not quandles, there can be other such tuples, and counting them gives invariants that depend on more than the underlying set of~$R$.
\end{example}

\begin{remark}\label{rem:inj}
A slight variant of the preceding theory allows us to give another proof of Theorem~\ref{thm:inj_on_connected}. If we have~$b(C)=b(D)$, then~$C$ and~$D$ have the same additive invariants. In particular, from Example~\ref{ex:cardinality}, we see that they have the same number of elements. We now note that the assignment~$R\longmapsto|\Inj_\calR(C,R)|$ that sends any rack~$R$ to the number of~\textit{injective} morphisms~$C\to R$ is an additive invariant. This is again clear from Lemma~\ref{lem:connected_images}: if~$R$ is decomposed into~$S$ and~$T$, then every morphism into~$R$ lands in either~$S$ or~$T$, and injectivity is preserved. As the identity is an element in~$\Inj_\calR(C,C)$, we deduce that~$\Inj_\calR(C,D)$ is not empty, and any element~$C\to D$ in it is an injective morphism. As both racks are finite, it is necessarily an isomorphism.
\end{remark}

\begin{remark}\label{rem:knots}
Quandles are famous for being complete invariants of knots. We can now use Proposition~\ref{prop:marks} to turn things around and use knot quandles to produce invariants of racks and quandles! Knot quandles are finitely generated~(by the arcs in any diagram) and connected~\cite[Cor.~15.3]{Joyce}, so every knot~$K$ defines an additive invariant~\hbox{$\Phi_K\colon\rmB(\calR)\to\bbZ$} via its finitely generated knot quandle~$C=\rmQ(K)$.
\end{remark}


If we let~$C$ vary through a set of representatives of isomorphism classes of finitely generated connected racks, we get a morphism
\begin{equation}\label{eq:marks}
\prod_C\Phi_C\colon\rmB(\calR)\longrightarrow\prod_C\bbZ
\end{equation}
of commutative rings. Note that there are infinitely many isomorphism classes of finitely generated connected racks. Therefore, the product on the right-hand side of~\eqref{eq:marks} is not finite, and the image is not contained in the direct sum. The image of the class~$[\star]$ of the terminal rack~$\star$ is the constant family with value~$1$. It turns out that the homomorphism~\eqref{eq:marks} is injective, and it is even enough to restrict to marks defined using finite~$C$:

\begin{theorem}\label{thm:marks}
If~$C$ ranges through a set of representatives of isomorphism classes of finite connected racks, then the morphism~\eqref{eq:marks} of commutative rings is injective.
\end{theorem}

\begin{proof}
First, building on Remark~\ref{rem:inj}, we argue that the numbers~$|\Mor_\calR(C,R)|$, for varying~$C$, determine the numbers~$|\Inj_\calR(C,R)|$ of \textit{injective} rack morphisms~\hbox{$C\to R$}, for varying~$C$. To see this, note that a morphism~$C\to R$ is either injective or has an image that has fewer elements than~$C$:
\begin{equation}\label{eq:split_off_inj}
|\Mor_\calR(C,R)|=|\Inj_\calR(C,R)|+\sum_{|D|<|C|}|\Mor_\calR^D(C,R)|,
\end{equation}
where~$\Mor_\calR^D(C,R)$ is the set of morphisms~$C\to R$ with image isomorphic to~$D$. 

We can write every element of~$\Mor_\calR^D(C,R)$ as a composition~$C\to D\to R$ where the first morphism is surjective and the second morphism is injective, and the group~$\Aut(D)$ acts freely and transitively on the representations~(identifications of~$D$ with the image). Therefore, we have a bijection
\begin{equation}\label{eq:Mor^D_as_twisted_product}
\Mor_\calR^D(C,R)\cong\Inj_\calR(D,R)\times_{\Aut(D)}\Sur_\calR(C,D),
\end{equation}
where we have written~$\Sur_\calR(C,D)$ for the set of surjective morphisms~\hbox{$f\colon C\to D$} and the index~$\Aut(D)$ means that we identify a pair~$(hg,f)$ with the pair~$(h,gf)$ whenever~$g$ is an automorphism of~$D$ and~$h$ is an injection~$D\to R$. Under the bijection, the class~$[hg,f]=[h,gf]$ on the right corresponds to the morphism~$hgf$ on the left. 

By induction, we see that the numbers~$\Inj_\calR(D,R)$ for~$|D|<|C|$ are determined by the numbers~$|\Mor_\calR(D,R)|$, and we see from~\eqref{eq:split_off_inj} and~\eqref{eq:Mor^D_as_twisted_product} that, together with~$|\Mor_\calR(C,R)|$, these  determine the~$|\Inj_\calR(C,R)|$.

We now take an element in the kernel of~\eqref{eq:marks} and write it in the form
\[
\sum_{[D]}\lambda_Db(D),
\]
where~$D$ ranges over the isomorphism classes of connected finite racks. By the first part, we have
\[
0=\sum_{[D]}\lambda_D|\Inj_\calR(C,D)|
\]
for all connected finite racks~$C$. There is a partial order on the set of isomorphism classes of connected finite racks, given by~$[C]\leqslant[D]$ if and only if~$\Inj_\calR(C,D)\not=\emptyset$. Assume that one of the coefficients was not zero. Then we could choose~$[C]$ maximal among those with~$\lambda_C\not=0$. As we have~$|\Inj_\calR(C,D)|=0$ for~$[C]\not\leqslant[D]$, this gives
\[
0=\sum_{[D]}\lambda_D|\Inj_\calR(C,D)|
=\sum_{[C]\leqslant[D]}\lambda_D|\Inj_\calR(C,D)|.
\]
By the maximality of~$C$, this is
\[
\lambda_C|\Inj_\calR(C,C)|=\lambda_C\Aut(C)\not=0,
\]
a contradiction. Therefore, all coefficients must be zero.
\end{proof}

\begin{remark}
Theorems~\ref{thm:basis} and~\ref{thm:marks} together give a~$\bbZ$--linear injection
\[
\bigoplus_C\bbZ\longrightarrow\prod_C\bbZ
\]
which is \textit{not} the inclusion.
\end{remark}


\section{Permutations racks}\label{sec:perm}

We will now look at permutation racks in more detail and explain how the Dress--Siebeneicher theory of the Burnside ring of the infinite cyclic group, developed in~\cite{Dress--Siebeneicher:pro-finite} and~\cite{Dress--Siebeneicher:Witt}, fits into our context.

\begin{proposition}\label{prop:permutations}
A permutation rack~$(R,\pi)$ is connected if and only if~$\pi$ is a transitive cycle,~i.e., one that involves all elements of~$R$. A permutation rack~$(R,\pi)$ is homogeneous if and only if~$\pi$ is a product of disjoint cycles, each of the same length. A permutation rack~$(R,\pi)$ is irreducible if and only if~$\pi\not=\id_R$ or~$R$ has at most one element.
\end{proposition}

\begin{proof}
The inner automorphism group of a permutation rack~$(R,\pi)$ is the subgroup of the symmetric group~$\Sym(R)$ generated by the permutation~$\pi$. This subgroup acts transitively on~$R$ if and only if~$\pi$ is transitive, i.e., it is a cycle involving all elements of~$R$.

The automorphism group of a permutation rack~$(R,\pi)$ is given by the centraliser of the permutation~$\pi$ in the symmetric group~$\Sym(R)$. 
If a permutation~\hbox{$\pi=\pi_1\cdots\pi_n$} is a product of cycles~$\pi_j$ of the same length, then the automorphism group~\hbox{$\Aut(R,\pi)=\langle\pi_1\rangle\wr\Sym(n)$} is the wreath product of the cyclic group generated by one cycle and the group that permutes the cycles~$\pi_1,\dots,\pi_n$.~(A special case is~$\pi=\id$, when all cycles~$\pi_j$ have length~$1$.) 
In general, any permutation~$\pi$ decomposes as a product of disjoint cycles which can be grouped into conjugacy classes, and~$\Aut(R,\pi)$ is the direct product of the corresponding wreath products. 
In particular, the group~$\Aut(R,\pi)$ acts transitively on the set~$R$ if and only if~$\pi$ is a product of cycles of the same length.

If~$\pi=\id_R$ and~$R$ has at least two elements, then~$R$ is clearly reducible. Assume now that~$\pi\not=\id_R$ is a non-identity permutation of~$R$, and in particular, we have~\hbox{$R\not=\emptyset$}, and there is an~$a$ in~$R$ such that~\hbox{$\pi(a)\not=a$}. It follows from~\eqref{eq:eqrel} that this~$a$ is equivalent to all other elements~$b$ in~$R$, and the rack~$(R,\pi)$ cannot be written as a disjoint union in any non-trivial way.
\end{proof}


Burnside rings can also be defined for pro-finite groups~$G$, and in particular, for the pro-finite completion~$\widehat{\bbZ}$ of the infinite cyclic group~(see~\cite{Dress--Siebeneicher:pro-finite} and~\cite{Dress--Siebeneicher:Witt}). As we only work with finite~$G$--sets, there is a Burnside ring~$\rmB(\bbZ)$ of finite~$\bbZ$--sets, which can be thought of as pairs~$(X,\pi)$, where~$X$ is a finite set and~$\pi$ is a permutation of~$X$. The ring~$\rmB(\bbZ)$ has an integral basis consisting of the classes~$c_n$ of the~$n$--cycles, with~$n\geqslant 1$. Multiplicatively, we have the relation~\hbox{$c_r\cdot c_s=\gcd(r,s)\cdot c_{\lcm(r,s)}$}. In particular, the element~$c_1$ is the identity. The unit~$\bbZ\to\rmB(\bbZ)$ admits a retraction (or augmentation) by the morphism~$\rmB(\bbZ)\to\bbZ$ that sends the class of a permutation to the number of its elements. 

\begin{remark}
The ring structure of the Burnside ring~$\rmB(\bbZ)$ features some surprises. First of all, it has zero-divisors. For instance, we obviously have~$c_p^2=pc_p$ for all prime numbers~$p$, which gives~\hbox{$c_p(c_p-p)=0$}. We also have~`unexpected' units, such as~$1-c_2$, because of~\hbox{$(1-c_2^2)=1-2c_2+c_2^2=1$}. There are no idempotent elements in~$\rmB(\bbZ)$, though, as any element in there comes from~$\rmB(\bbZ/n)$ for some~$n$, and Dress~\cite{Dress} showed that a Burnside ring~$\rmB(G)$ of a finite group~$G$ has non-trivial idempotents if and only if the group~$G$ is not solvable. In particular, it would have to be non-abelian.
\end{remark}

\begin{proposition}\label{prop:permutations_split}
There are morphisms
\begin{equation}\label{eq:permutations_split}
\rmB(\bbZ)\longrightarrow\rmB(\calR)\longrightarrow\rmB(\bbZ)
\end{equation}
of rings whose composition is the identity.
\end{proposition}

\begin{proof}
The elements of~$\rmB(\bbZ)$ are given by the isomorphism classes of finite permutation representations~$(X,\pi)$, i.e., finite sets~$X$ together with a permutation~$\pi$ on~$X$. A morphism~$\rmB(\bbZ)\to\rmB(\calR)$ of rings is given by sending the class of such a permutation to the class of the corresponding permutation rack~$(X,\pi)$, where~$w\rhd x=\pi(x)$ for all~$w,x$ in~$X$.

Conversely, given a finite rack~$R$, the canonical automorphism~$\sigma$ from~\eqref{eq:canonical_automorhphism} defines a permutation of~$R$. This is an additive invariant, as~$\sigma$ preserves any decomposition of~$R$ into subracks. Therefore, we find that~$(R,\rhd)\mapsto(R,\sigma)$ is an additive invariant, and it defines a morphism~$\rmB(\calR)\to\rmB(\bbZ)$.

The composition is the identity, as the canonical automorphism of a permutation rack is just the permutation:~$\sigma(x)=x\rhd x=\pi(x)$. 
\end{proof}

\begin{remark}
As any quandle, by definition, has a trivial canonical automorphism, the classes of all quandles in~$\rmB(\calR)$ map to the trivial summand~$\bbZ$ in~$\rmB(\bbZ)$ generated by the trivial permutations~(see Section~\ref{sec:B(Q)} again).
\end{remark}

\begin{remark}
The power operations defined in~\cite[Sec.~6]{Szymik:2018} for racks induce power operations~\hbox{$\Psi^n\colon\rmB(\calR)\to\rmB(\calR)$}, for~$n\in\bbZ$, which send the class of rack~$(R,\rhd)
$ to the class of~$(R,\rhd^n)$, where~$\rhd^n$ is the~$n$--fold iteration of~$\rhd$ in the obvious sense: the element~$x$ has to act~$n$--times on~$y$ to give~$x\rhd^n y$. The morphisms in~\eqref{eq:permutations_split} are compatible with these power operations. Note that for every finite rack~$(R,\rhd)$ there is a positive integer~$n$ such that~$\Psi^n(R,\rhd)=(R,\id_R)$ is trivial: take for~$n$ the exponent of the symmetric group of permutations of~$R$.
\end{remark}


\begin{remark}\label{rem:lambda}
As one application of Theorem~\ref{thm:basis}, we can now see that the power operations~$\Psi^n$ on the ring~$\rmB(\calR)$ do not come from a~$\lambda$--ring structure on~$\rmB(\calR)$. Otherwise, they would satisfy the Frobenius congruence~$\Psi^p(x)\equiv x^p$ mod~$p$ for all prime numbers~$p$. However, this is not the case. Take~$x$ to be the class of the connected quandle with three elements, as in Example~\ref{ex:Sym(3)}. Then~$x^2$ is represented by a connected quandle of order~$9$, and~$\Psi^2(x)$ is the class of the trivial quandle of order~$3$. We show that the difference is not divisible by~$2$. To do so, consider the homomorphism
\begin{equation}\label{eq:epsilon}
\epsilon\colon\rmB(\calR)\longrightarrow\bbZ
\end{equation}
of abelian groups that sends the class of~$\star$ to~$1$ and all other connected racks to~$0$. This is well-defined by Theorem~\ref{thm:basis}. From the definition, we immediately get~$\epsilon(\Psi^2(x))=\epsilon(3)=3$ and~$\epsilon(x^2)=0$. If the difference~$\Psi^2(x)-x^2$ were divisible by~$2$, then applying~$\epsilon$ we would find that~$\epsilon(\Psi^2(x)-x^2)$ would be even, and that is a contradiction. Note that we have~$\epsilon(b(R))=|\tau(R)|$ for \text{connected} racks~$R$, where~$\tau$ is from Example~\ref{ex:tau}. There, we showed that the right-hand side is~\textit{not} an additive invariant of racks, in contrast to the left-hand side. 
\end{remark}


\subsection{Profiles}\label{sec:profiles}

Recall from Remark~\ref{rem:profile} that the profile of a connected rack is defined as the cycle structure of any of its left multiplications. Racks with relatively simple profiles can be classified~(see Lopes--Roseman~\cite[Sec.~8]{Lopes--Roseman}). Here is what we can say in general.

\begin{proposition}\label{prop:profile}
The profile extends to a ring homomorphism
\begin{equation}\label{eq:profile}
\lambda\colon\rmB(\calR)\longrightarrow\rmB(\bbZ).
\end{equation}
\end{proposition}

\begin{proof}
First, let~$R$ be a connected rack. The profile is a conjugacy class in the permutation group~$\Sym(R)$ that we will encode as a linear combination.
\[
\lambda(R)=\sum_{n=1}^\infty\lambda_n(R)c_n\in\rmB(\bbZ)
\]
of the~$n$--cycles~$c_n$ in~$\rmB(\bbZ)$: the integer~$\lambda_n(R)$ counts the number of~$n$--cycles in left multiplications of~$R$. 

The element~$\lambda(R)$ depends only on the isomorphism class of the connected rack~$R$. As the isomorphism classes of connected racks form an integral basis for~$\rmB(\calR)$ by Theorem~\ref{thm:basis}, there is a unique homomorphism~\eqref{eq:profile} of abelian groups that extends the profile. It remains to be checked that this homomorphism~$\lambda$ is compatible with the multiplicative structure.

As the rack structure on a cartesian product~$R=S\times T$ is given component-wise, 
\[
(s,t)\rhd(s',t')=(s\rhd s',t\rhd t')
\]
we see that the left multiplication~$\ell_{(s,t)}=\ell_s\times\ell_t$ is the product of the left multiplications. As this corresponds precisely to the product in the Burnside ring~$\rmB(\bbZ)$ of permutations, we have shown that the equality~\hbox{$\lambda(S\times T)=\lambda(S)\cdot\lambda(T)$} holds for connected racks. It follows for all of~$\rmB(\calR)$ by linearity.
\end{proof}

\begin{remark}
It is clear that the profile~$\lambda\colon\rmB(\calR)\to\rmB(\bbZ)$ from Proposition~\ref{prop:profile} is very different from the retraction~$\rmB(\calR)\to\rmB(\bbZ)$ in Proposition~\ref{prop:permutations_split}. The latter sends all quandles to trivial permutations, whereas the profiles of quandles are only trivial for trivial quandles.
\end{remark}


\section{Prime quandles and products with cycles}\label{sec:prime}

In this section, we examine the multiplicative structure of the Burnside rings of racks and quandles in greater detail. We start with the following cancellation result. It very well illustrates the value of our novel perspective: the Burnside rings do not appear in the statement, but in the proof.

\begin{theorem}\label{thm:cancellation}
If~$R$ and~$S$ are connected finite racks, and~$T$ is a finite rack with an element~$t$ such that~$t\rhd t =t$~(for example, a non-empty quandle) such that there exists an isomorphism~\hbox{$R\times T\cong S\times T$}, then we have~$R\cong S$.
\end{theorem}

\begin{proof}
From~$R\times T\cong S\times T$, we get~$b(R)b(T)=b(R\times T)=b(S\times T)=b(S)b(T)$. By hypothesis, there is a morphism~$C\to T$ from every connected rack to~$T$. All marks of~$b(T)$ are non-zero, and this implies that~$b(T)$ is not a zero-divisor. Then we know~$b(R)=b(S)$, and~$R\cong S$ now follows from Theorem~\ref{thm:inj_on_connected}.
\end{proof}

\begin{remark}
The proof shows, in particular, that the class~$b(Q)$ of any finite, non-empty quandle~$Q$ is a non-zero element that is not a zero-divisor in the Burnside ring.
\end{remark}

\begin{example}\label{ex:bad_product}
The product structure established in Proposition~\ref{prop:AR_multiplication} does not play well with the integral basis consisting of the classes of connected racks: a product of connected racks does not need to be connected. For instance, if~$R$ is the non-trivial permutation rack on a two-element set, then~$R\times R$ is not connected; it can be decomposed into two copies of~$R$. Example~\ref{ex:Sym(3)} contains another instance.
\end{example}

The~cartesian product of two connected racks is connected if at least one factor is a quandle~(see~\cite[Lem.~1.20]{Andruskiewitsch--Grana}). This implies that the cartesian product defines an abelian monoid structure on the set~$\Qcon$ of isomorphism classes of finite quandles that are connected. Thus, the free abelian group~$\bbZ\{\Qcon\}$ with basis given by this set is a commutative ring, the monoid ring for this monoid. 

\begin{theorem}\label{thm:Burnside_for_quandles_is_monoidal}
There is an isomorphism
\begin{equation}\label{eq:additive_generators_Q}
\rmB(\calQ)\cong\bbZ\{\Qcon\}
\end{equation}
of rings between the Burnside ring of finite quandles and the monoid ring of the abelian monoid of isomorphism classes of connected quandles.
\end{theorem}

Before we prove this result, we give two remarks to provide more context for it.

\begin{remark}
The difference between Theorem~\ref{thm:Burnside_for_quandles_is_monoidal}, which is about quandles, and Theorem~\ref{thm:basis} for racks is that the latter features only an isomorphism of abelian groups, as there is no useful multiplication on the isomorphism classes of finite connected racks by Example~\ref{ex:bad_product}.
\end{remark}

\begin{remark}
The reader may wonder why not turn Theorem~\ref{thm:Burnside_for_quandles_is_monoidal} with isomorphism~\eqref{eq:additive_generators_Q} into the definition of the Burnside ring of finite quandles. This is possible, and then our definition would turn into the theorem that the decomposition into connected components is the universal additive invariant for finite quandles.
\end{remark}

\begin{proof}[Proof of Theorem~\ref{thm:Burnside_for_quandles_is_monoidal}]
In a fashion similar to the proof of Theorem~\ref{thm:basis}, we can construct a homomorphism~\hbox{$\bbZ\{\Qcon\}\to\rmB(\calQ)$} and prove that it is an isomorphism. It remains to be observed that under this isomorphism, the multiplications on both sides correspond: both are derived from the cartesian product.
\end{proof}


\subsection{Prime quandles}

Theorem~\ref{thm:Burnside_for_quandles_is_monoidal} leaves us with the problem of understanding the structure of the abelian monoid~$\Qcon$ with respect to the cartesian product. The following definition appears to be new.

\begin{definition}\label{def:prime}
We call a connected quandle~\textit{prime} if it is not a singleton and if it can only be written as a product in trivial ways---with one factor a singleton.
\end{definition}

\begin{examples}
As there are no connected quandles with two elements, the smallest prime quandle has three elements~(see Example~\ref{ex:simple}). It follows that all connected quandles with fewer than nine elements are prime. One of these is the tetrahedral quandle of order~$4$, which consists of one of the conjugacy classes of the~$3$--cycles in the alternating group of order~$12$. Of course, every connected quandle with a prime number of elements is prime. These have been classified~\cite{ESG}, and there are~\hbox{$p-2$} isomorphism classes of prime quandles of order~$p$: the conjugacy classes of size~$p$ in the holomorph~\hbox{$\bbZ/p\rtimes(\bbZ/p)^\times$}, i.e., the affine group of the~$1$--dimensional vector space over the prime field~$\bbF_p$. We see that there are infinitely many prime quandles.
\end{examples}

Let~$\primeQ$ be the subset of~$\Qcon$ that consists of the isomorphism classes of prime quandles.

\begin{theorem}\label{thm:multiplicative_generators}
There is a surjective ring homomorphism
\[
\bbZ[\primeQ]\longrightarrow\rmB(\calQ)
\]
from the polynomial ring~$\bbZ[\primeQ]$ on the set~$\primeQ$ of variables onto the Burnside ring of finite quandles.
\end{theorem}

\begin{proof}
Let~$\bbN\{\primeQ\}$ be the free abelian monoid with basis~$\primeQ$, written multiplicatively. Then, by the universal property, there is a morphism
\begin{equation}\label{eq:abmon}
\bbN\{\primeQ\}\longrightarrow\Qcon,\;\prod_j[P_j]^{n_j}\longmapsto\bigg[\prod_jP_j^{n_j}\bigg],
\end{equation}
of abelian monoids that extends the inclusion~$\primeQ\to\Qcon$ of the subset~$\primeQ$ into the abelian monoid~$\Qcon$. By induction on the number of elements, every connected quandle is isomorphic to a product of prime quandles. This shows that the morphism~\eqref{eq:abmon} of abelian monoids is surjective.

The surjective morphism~\eqref{eq:abmon} of abelian monoids induces a surjective morphism
\[
\bbZ[\primeQ]=\bbZ\{\bbN\{\primeQ\}\}\longrightarrow\bbZ\{\Qcon\}
\]
of rings from the polynomial ring~$\bbZ[\primeQ]$ on the set~$\primeQ$ of variables, which is the monoid ring of the abelian monoid~$\bbN\{\primeQ\}$. Composition with the isomorphism~\eqref{eq:additive_generators_Q} from Theorem~\ref{thm:Burnside_for_quandles_is_monoidal} gives the result.
\end{proof}


\begin{remark}\label{rem:uniqueness}
It would be interesting to decide whether or not the morphism in Theorem~\ref{thm:multiplicative_generators} is injective as well. If this were the case, the Burnside ring of quandles would be polynomial in a set of variables that corresponds to the prime quandles.
\end{remark}


\subsection{Products with cycles}

Let~$\Pcon$ denote the set of isomorphism classes of connected permutation racks. Recall from Proposition~\ref{prop:permutations} that these are represented by~$n$--cycles for~$n\geqslant 1$. As before, we let~$\Qcon$ and~$\Rcon$ denote the sets of isomorphism classes of connected quandles and racks, respectively.

\begin{theorem}\label{thm:CxQ<R}
The cartesian product induces an injection
\[
\Pcon\times\Qcon\longrightarrow\Rcon,\,([C],[Q])\longmapsto[C\times Q].
\]
\end{theorem}

\begin{proof}
The product of two connected racks is connected if at least one of the factors is a quandle~(see~\cite[Lem.~1.20]{Andruskiewitsch--Grana} again). This shows that the map is well-defined. To show that it is injective, we need to recover the isomorphism types of the cycle~$C$ and the quandle~$Q$ from the rack~$C\times Q$.

The canonical automorphism~$\sigma_{X\times Y}$ of a product of racks is given by the cartesian product~\hbox{$\sigma_X\times\sigma_Y$} of the canonical automorphisms:
\[
\sigma_{X\times Y}(x,y)=(x,y)\rhd(x,y)=(x\rhd x,y\rhd y)=(\sigma_X(x),\sigma_Y(y)).
\]
If~$X=C$ is a transitive cycle, then~$\sigma_C$ is just that transitive cycle. If~$Y=Q$ is a quandle, then~$\sigma_Q=\id_Q$ is the identity. Thus, we see that~$\sigma_{C\times Q}=\sigma_C\times\id_Q$ decomposes into as many cycles of type~$C$ as~$Q$ has elements. This determines the cycle type~$C$. For the isomorphism type of the quandle~$Q$, we note that the argument above shows that the projection~$C\times Q\to Q$ induces an isomorphism between~$Q$ and the set of~$\sigma_{C\times Q}$--orbits of~$C\times Q$. In other words, we can recover~$Q$ as the associated quandle of the rack~$C\times Q$, which is the value of the left-adjoint to the forgetful functor from the category of quandles to the category of racks.
\end{proof}


\begin{example}\label{ex:not_surjective}
The map in Theorem~\ref{thm:CxQ<R} is not surjective. To see this, we can take the group~\hbox{$G=\mathrm{SL}_2(\mathbb{F}_3)$} of order~$24$. The group~$G$ acts transitively on the set~$R$ of eight non-zero vectors in~$\mathbb{F}_3^2$. The stabiliser~$H$ of the first standard basis vector is the subgroup of order~$3$ that consists of the upper-triangular matrices with~$1$'s on the diagonal. Thus, we can identify~$R$ with~$G/H$; a coset~$bH$ corresponds to the first column of the matrix~$b$. As in~\eqref{eq:like-x-set}, a connected rack structure on this set is given by~$aH\rhd bH=ama^{-1}bH$ with the following matrix~$m$, which commutes with all~$h\in H$.
\[
m=\begin{bmatrix}
-1 & -1 \\ \phantom{-}0 & -1
\end{bmatrix}
\]
Then~$aH$ acts on the vector~$bH$ by left-multiplication with the matrix~$ama^{-1}$, which only depends on the coset~$aH$. Note that the actions of~$aH$ and~$-aH$ are equal, as~$ama^{-1}=(-a)m(-a)^{-1}$ so that the eight vectors only produce four different matrices. Since~$aH\rhd aH=amH$, we see that the canonical automorphism is given by right-multiplication with~$m$, which is a non-trivial involution, since~$m\not\in H$ but~$m^2\in H$. The quotient is the associated quandle, which is connected and has order~$4$, so it must be isomorphic to the tetrahedral quandle. However, the rack~$R$ is not isomorphic to the product of the tetrahedral quandle and the cyclic rack of order~$2$, as the inner automorphism group of~$R$ is isomorphic to~$\mathrm{SL}_2(\mathbb{F}_3)$. In contrast, the inner automorphism group of that product is the product~\hbox{$\mathrm{A}_4\times\mathrm{C}_2$}. We refer to~\cite[Ex.~8.6]{HSV} for more information related to~$R$.
\end{example}


\begin{corollary}\label{cor:tensor_embed}
The cartesian product induces an injective morphism
\[
\rmB(\bbZ)\otimes\rmB(\calQ)\longrightarrow\rmB(\calR)
\]
of rings.
\end{corollary}

\begin{proof}
This follows from Theorem~\ref{thm:CxQ<R} as~$\Pcon$ is an integral basis for~$\rmB(\bbZ)$, whereas~$\Qcon$ and~$\Rcon$ are bases for~$\rmB(\calQ)$ and~$\rmB(\calR)$.
\end{proof}

\begin{remark}
By Theorem~\ref{thm:basis}, the Burnside ring~$\rmB(\calR)$ of racks is free as an abelian group. Its subring~$\rmB(\bbZ)\otimes\rmB(\calQ)$ is evidently free as a module over~$\rmB(\bbZ)$, over~$\rmB(\calQ)$, and even over~$\rmB(\bbZ)\otimes\rmB(\calQ)$. It would be interesting to describe the Burnside ring~$\rmB(\calR)$ of racks as a module over these larger subrings.
\end{remark}


\section{Crossed Burnside rings}\label{sec:crossed}

In this section, we will first review the theory of crossed~$G$--sets and their Burnside rings, and then we will relate this theory to racks and their Burnside rings.


Let~$G$ be a group. We will denote the category of~$G$--sets by~$\calS^G$. For every~$G$--set~$Y$, we have the category~$\calS^G\downarrow Y$ of~$G$--sets over~$Y$. If~$Y=\star$ is the~$G$--set supported on a singleton, the category~$\calS^G\downarrow \star$ is equivalent to~$\calS^G$. The following case is most important for us.

\begin{definition}
If~$Y=\Ad(G)$ is the~$G$--set supported on the set~$G$ where the group~$G$ acts via conjugation, the category~\hbox{$\calX_G=\calS^G\downarrow\Ad(G)$} is the category of~\textit{crossed~$G$--sets}. In other words, a crossed~$G$--set is a~$G$--set~$X$ together with a~$G$--map~\hbox{$\delta\colon X\to\Ad(G)$}, and so equivariance of~$\delta$ means~\hbox{$\delta(gx)=g\delta(x)g^{-1}$}. 
\end{definition}

\begin{remark}
One motivation for introducing crossed~$G$--sets is the fact that the category of crossed~$G$--sets is the Drinfeld centre of the monoidal category of~$G$--sets with respect to the cartesian product: we have~\hbox{$\calZ(\calS^G,\times,\star)\simeq \calX_G$}~\cite{Freyd--Yetter,Krahmer--Wagemann}. In particular, the category of crossed~$G$--sets is braided monoidal. We refer to the sequel~\cite{Mazza--Szymik:Mackey} for more on centres in this context.
\end{remark}


\begin{remark}
The category of~$G$--sets is, of course, a functor category, i.e., it is equivalent to a category of functors from some category to the category~$\calS$ of sets: it is the category of functors from~$G$, thought of as a category with one object, to the category of sets. The category of~$G$--sets over~$Y$ is also equivalent to a functor category. This time, the objects of the indexing category are the elements of the set~$Y$, and the morphisms~$a\to b$ are the group elements~$g$ with~$ga=b$. In particular, we can use this observation to rewrite the category of crossed~$G$--sets as a functor category: take~$Y=\Ad(G)$. Now the objects are the elements of~$G$ and the morphisms~$a\to b$ are the elements~$g$ in~$G$ such that~$gag^{-1}=b$.~(This can also be thought of as the automorphism category of the category associated with~$G$.) In particular, the category of crossed~$G$--sets---for a fixed group~$G$---is a topos.
\end{remark}

\begin{remark}
The disjoint union defines the categorical sum of~$G$--sets over~$Y$. Products of~$G$--sets over~$Y$ are more subtle, as they are given by pullbacks, which do not respect the underlying sets unless~$Y=\star$ is terminal. However, when~$Y=\Ad(G)$, the group multiplication~\hbox{$(x,y)\mapsto x\cdot y$} within~$G$ induces a~$G$--map~\hbox{$\Ad(G)\times\Ad(G)\to\Ad(G)$}, because
\[
(gxg^{-1})\cdot(gyg^{-1})=g(x\cdot y)g^{-1}.
\]
We can use that to define a monoidal category structure that respects the underlying sets.
\end{remark}


For a given~$G$--set~$Y$, we can form the Grothendieck group~$\rmB(G;Y)$ of the category of finite~$G$--sets over~$Y$. This gives the universal functorial additive invariant of finite~$G$--sets in the sense of Remark~\ref{rem:functorial}. For~$Y=\star$, we get the universal additive invariant~$\rmB(G;\star)=\rmB(G)$, the ordinary Burnside ring of~$G$. More generally, for a transitive~$G$--set~\hbox{$Y=G/H$}, we have an isomorphism
\[
\rmB(G;G/H)\cong\rmB(H),
\]
as every~$G$--set over~$G/H$ comes from an~$H$--set by induction. This isomorphism, together with the isomorphism~$\rmB(G;Y\sqcup Z)\cong\rmB(G;Y)\oplus\rmB(G;Z)$ allows us to compute all~$\rmB(G;Y)$ in terms of Burnside rings.

\begin{definition}
The~\textit{crossed Burnside ring} of~$G$ is
\[
\rmB^\times(G)=\rmB(G;\Ad(G)).
\]
\end{definition}

Suitable references for crossed Burnside rings are~\cite{Oda--Yoshida, Bouc:2003b, Mazza}. From the above generalities, we get
\[
\rmB^\times(G)\cong\bigoplus_{[g]}\rmB(\rmC_G(g)),
\]
where~$[g]$ ranges over the conjugacy classes of elements~$g$ in~$G$ and~$\rmC_G(g)$ is the centraliser of~$g$ in~$G$. In particular, the neutral element~$e$ gives rise to a copy of the Burnside ring~$\rmB(G)$ inside~$\rmB^\times(G)$, whose classes are represented by crossed~$G$--sets of the form~\hbox{$X\to\{e\}=\star\to\Ad(G)$}. An integral basis for the crossed Burnside ring~$\rmB^\times(G)$ of~$G$ is given by the~set of isomorphism classes~$[G/H,a]$ of transitive crossed~$G$--sets~$(G/H,a)$, where~$H$ is a subgroup~$H$ of~$G$ and~$a$ is an element of~$G$ that commutes with every element in~$H$. As a crossed~$G$--set, this is the homogeneous~$G$--set~$G/H$ with the~`crossing' given by
\[
G/H\longrightarrow\Ad(G),\,gH\mapsto gag^{-1}.
\] 
Two transitive crossed~$G$--sets~$(G/H,a)$ and~$(G/K,b)$ are isomorphic if and only if there exists an element~\hbox{$g\in G$} such that~$K=gHg^{-1}$ and~$b=gag^{-1}$.


\subsection{Functoriality}

Any morphism~$f\colon G\to H$ between groups gives rise to a restriction functor~\hbox{$f^*\colon\calS^H\to\calS^G$} with~$f^*X=X$. This functor can be thought of as a right Kan extension. By left Kan extension, a morphism~$f\colon G\to H$ of groups gives rise to an induced functor~\hbox{$f_*\colon\calS^G\to\calS^H$} with~$f_*X=H\times_GX$. Here, the set~\hbox{$H\times_GX$} is the set of equivalences classes for the equivalence relation on~$H\times X$ given by~\hbox{$(hf(g),x)\sim(h,gx)$}. This set is an~$H$--set for the left multiplication on the first coordinate. We write~$[h,x]$ for the equivalence class of~$(h,x)$.

For the~$G$--set~$X=\Ad(G)$, with the action given by conjugation, we have a canonical~$H$--map~\hbox{$f_*\Ad(G)\to\Ad(H)$}, defined by~\hbox{$[h,x]\mapsto hf(x)h^{-1}$}. Then 
\[
f_*([hf(g),x])=f_*([h,gx])=hf(gxg^{-1})h^{-1},
\] 
and this shows that the map is well-defined on~$f_*\Ad(G) = H \times_G \Ad(G)$. This leads to the following result.

\begin{lemma}\label{lem:fibration}
A morphism~$f\colon G\to H$ of groups induces a functor 
\[
f_*\colon\calX_G\longrightarrow\calX_H.
\]
It sends a crossed~$G$--set~$\delta\colon X\to\Ad(G)$ to the crossed~$H$--set
\[
f_*X\longrightarrow f_*\Ad(G)\longrightarrow\Ad(H)
\]
where the first map~$f_*\delta$ comes from the functoriality of induction and the second map is the canonical map~\hbox{$f_*\Ad(G)\to\Ad(H)$}.
\end{lemma}

\begin{remark}
The functoriality of the crossed Burnside ring~$\rmB^\times(G)$ is lamen\-table. We recall that the ordinary Burnside ring~$G\mapsto\rmB(G)$ gives a Mackey functor, even a Tambara functor, and in particular, the restriction along any group homomorphism~\hbox{$f\colon G \to H$} induces a homomorphism~$f^\star\colon\rmB(H)\to\rmB(G)$ of rings. In contrast, the only known ring non-trivial homomorphisms between crossed Burnside rings are the Brauer morphisms of Bouc~\cite[2.3]{Bouc:2003b}, which for a subgroup~\hbox{$H\leqslant G$} give a homomorphism~\hbox{$\rmB^\times(G)\to\rmB^\times(\rmC_G(H))$} of rings via sending a crossing~$\delta\colon X\to\Ad(G)$ to~\hbox{$\delta^H\colon X^H\to\Ad(G)^H=\Ad(\rmC_G(H))$}, where~$\rmC_G(H)$ is the centraliser of~$H$ in~$G$. 
The subgroups~$H$ and~$\rmC_G(H)$ are, in general, incomparable, as witnessed, for example, by any one of the two non-abelian subgroups of index~$2$ in the dihedral group of order~$12$: they are both dihedral of order~$6$, and their centraliser is cyclic of order~$2$, but it is not contained in them.
\end{remark}


\subsection{The relation to the Burnside ring of finite racks}\label{sec:X_to_R}

We can now describe the Burnside ring~$\rmB(\calR)$ of finite racks in terms of crossed Burnside rings of finite groups. 

By Proposition~\ref{prop:universal_additive_invariant}, defining morphisms~\textit{out of} the abelian group~$\rmB(\calR)
$ is equivalent to finding additive invariants of finite racks. Here, on the contrary, is a way to construct morphisms~\textit{into} it.

\begin{proposition}\label{prop:crossed_to_Burnside_of_racks}
For each finite group~$G$, there is a morphism
\begin{equation}\label{eq:crossed_G_to_racks}
\rmB^\times(G)\longrightarrow\rmB(\calR)
\end{equation}
of rings.
\end{proposition} 

\begin{proof}
Any crossed~$G$--set~$\delta\colon X\to G$ defines a rack structure on~$X$ via the formula~\hbox{$a\rhd b=\delta(a)b$}. This construction is compatible with isomorphisms, sums, and products.
\end{proof}

\begin{remark}
For any nontrivial group~$G$, the morphism~\eqref{eq:crossed_G_to_racks} of rings sends the ordinary Burnside ring~\hbox{$\rmB(G)\subseteq\rmB^\times(G)$} to the summand~$\bbZ\subseteq\rmB(\calR)$ corresponding to the trivial racks, where~$x\rhd y=y$. Therefore, the augmentation ideal of~$\rmB(G)$, i.e., the kernel of the restriction morphism~$\rmB(G)\to\rmB(e)=\bbZ$, is always in the kernel of~\eqref{eq:crossed_G_to_racks}. It follows that the morphism~\eqref{eq:crossed_G_to_racks} is only ever injective if~$G$ is the trivial group.
\end{remark}  

\begin{example}
Recall that the crossed Burnsidering~$\rmB^\times(G)$ has an integeral basis given by elements~$[H,a]$ with representatives of the form~$G/H\to\Ad(G)$ with~$xH\mapsto xax^{-1}$ for some element~\hbox{$a\in\rmC_G(H)$}. The corresponding rack structure on the set~$G/H$ of cosets is given by the formula~\hbox{$xH\rhd yH=xax^{-1}yH$}. In particular, if~$G$ is an abelian group, where conjugation is trivial, the image has~\hbox{$xH\rhd yH=ayH$}, and the result is a permutation rack. In other words, abelian groups see nothing of~$\rmB(\calR)$ that is not already contained in the image of the morphism~$\rmB(\bbZ)\to\rmB(\calR)$ from Proposition~\ref{prop:permutations_split}.
\end{example}


It is evident from Proposition~\ref{prop:permutations_split} that none of the morphisms~\eqref{eq:crossed_G_to_racks} can be surjective; the source is a finitely generated abelian group. 

\begin{proposition}\label{prop:jointly_surjective}
The sum of the morphisms~\eqref{eq:crossed_G_to_racks}
\[
\bigoplus_G\rmB^\times(G)\longrightarrow\rmB(\calR)
\] 
is surjective. 
\end{proposition}

\begin{proof}
Every finite rack~$R$ is in the image of the morphism~\eqref{eq:crossed_G_to_racks} for~$G=\Aut(R)$, the finite automorphism group of the rack~$R$. The map~\hbox{$\delta\colon R\to\Aut(R)$} that sends~$a$ to the left-multiplication~$\ell_a$ defines a crossed~$\Aut(R)$--set structure on~$R$, and this crossed~$\Aut(R)$--set maps to~$R$ under~\eqref{eq:crossed_G_to_racks}. 
\end{proof}


\subsection{Globalisation}\label{sec:globalisation}

In this section, we discuss a `global' variant of the category of crossed~$G$--sets, the category~$\calX$ of global actions. It can be used to present the theory of racks as a localisation~(see~\cite{Lawson--Szymik:2}), and we will use it to give an alternative description of the Burnside ring of racks in terms of a global version of the crossed Burnside ring~(see~Theorem~\ref{thm:iso_BX_BR} below).

The Grothendieck construction takes the~(pseudo) functor~$G\mapsto\calX_G$~(see Lemma~\ref{lem:fibration}) from the category of groups to the category of categories to another category, say~$\calX$, that now fibres over the category of groups, with fibre over~$G$ equivalent to~$\calX_G$, the category of crossed~$G$--sets. In concrete terms, the objects of~$\calX$ can be identified with crossed~$G$--sets~$\delta\colon X\to\Ad(G)$ for some group~$G$ that is part of the data of the object. A morphism from a crossed action~$\delta\colon X\to\Ad(G)$ to another crossed action~$\epsilon\colon Y\to\Ad(H)$ is a pair~$(f,w)$ consisting of a morphism~\hbox{$f\colon G\to H$} of groups and a map~\hbox{$w\colon X\to f^*Y$} of~$G$--sets, where~$G$ acts on~$f^*Y$, which is~$Y$ as a set, via~$f$, such that the diagram
\[
\xymatrix{
X\ar[d]_w\ar[r]^-\delta &\Ad(G)\ar[d]^-{\Ad(f)}\\
f^*Y \ar[r]_-{f^*\epsilon} &f^*\Ad(H)
}
\]
of~$G$--sets commutes. Of course, this just means~$f\delta(x)=\epsilon w(x)$ for all~$x$ in~$X$.

\begin{definition}
We call~$\calX$ the category of~\textit{crossed actions}.
\end{definition}

\begin{remark}
The reader might be tempted to use the term~\textit{crossed sets}, without a~$G$. However, one of our primary references already defined crossed sets to be a special kind of quandles~(see~\cite[Def.~1.1]{Andruskiewitsch--Grana}), and it seems prudent to avoid redefining the term.
\end{remark}

A morphism~$(f,w)$ in the category~$\calX$ is an isomorphism if and only if~$f$ and~$w$ are bijective. The following weaker notion is important:

\begin{definition}\label{def:equivalence}
A morphism~$(f,w)$ of crossed actions is an~\textit{equivalence} if and only if the map~$w$ is a bijection.
\end{definition}

\begin{remark}
The equivalences are precisely the cartesian morphisms in the category~$\calX$ when considered fibred over the category of groups as above. 
\end{remark}


The category~$\calX$ of crossed actions has a terminal object~$\star\to\Ad(e)$, where~$\star$ is a singleton, and an initial object~$\emptyset\to\Ad(e)$. In both cases, the crossings are uniquely determined by the fact that~$\Ad(e)$ has a unique element. 

Let~$\delta\colon X\to\Ad(G)$ and~$\epsilon\colon Y\to\Ad(H)$ be two crossed actions with groups~$G$ and~$H$, respectively. Their product is defined as 
\begin{equation}\label{eq:product_of_crossed}
X\times Y\longrightarrow\Ad(G\times H),\,(x,y)\mapsto(\delta(x),\epsilon(y)).
\end{equation}
In contrast, the sum is more delicate and requires a more extended discussion.

Let~$\delta\colon X\to\Ad(G)$ and~$\epsilon\colon Y\to\Ad(H)$ be two crossed actions with groups~$G$ and~$H$, respectively. Their sum is supported on the disjoint union~$X\sqcup Y$. The action of the sum~$G*H$~(the~`free product') of the groups on~$X$ is via the given~$G$--action and the trivial action of~$H$. For~$Y$, it is the other way around: the group~$H$ acts as given, and~$G$ acts trivially. In other words, an element~\hbox{$g_1h_1\cdots g_nh_n$} acts on~$X$ as~\hbox{$g_1\cdots g_n$} and on~$Y$ as~\hbox{$h_1\cdots h_n$}. We define
\[
\delta*\epsilon\colon X\sqcup Y\to\Ad(G*H)
\] 
to be~$\delta$ on~$X$ and~$\epsilon$ on~$H$, composed with the canonical embeddings of~$G$ and~$H$ into the sum~$G*H$. Equivariance follows then from the equivariance of~$\delta$ and of~$\epsilon$. 

\begin{remark}
Joyce~\cite[Sec.~9]{Joyce} describes general colimits in the category of crossed actions. As the fibration from the category of crossed actions to the category of groups has a right adjoint, colimits can be formed by first forming the colimit~$G$ of the acting groups and then changing bases to define a new diagram that is defined over this colimit group~$G$. Then, we can form the usual colimit within the category of crossed~$G$--sets, which is just the colimit in the category of sets with the usual~$G$--action.
\end{remark}

Even if~$G$ and~$H$ are finite groups, the group~$G*H$ is usually infinite. Therefore, it will be beneficial for some purposes to have a description of a crossed action that is equivalent to the sum, in the sense given by Definition~\ref{def:equivalence}, and where the acting group is finite. This can be done as follows. The underlying set is, of course, the union~$X\sqcup Y$. On it, the product~\hbox{$G\times H$} acts on~$X$ and~$Y$ via the projections, so the disjoint union~$X\sqcup Y$ becomes a~$(G\times H)$--set, and 
\begin{align*}
X\sqcup Y&\longrightarrow\Ad(G\times H),\\
x&\longmapsto(\delta(x),e),\\
y&\longmapsto(e,\epsilon(y))
\end{align*}
defines a crossing. 

\begin{proposition}\label{prop:sum_of_crossed}
This crossed action is equivalent, in the sense given by Definition~\ref{def:equivalence}, to the sum constructed before. 
\end{proposition}

\begin{proof}
An equivalence betweem these crossed actions is given by the identity on~$X\sqcup Y$ and the canonical morphism~\hbox{$f\colon G*H\to G\times H$} of groups, given by the formula
\[
f(g_1h_1\cdots g_nh_n)=(g_1\cdots g_n,h_1\cdots h_n),
\] 
or more conceptually, using the universal property of the sum of groups by pairing the canonical embeddings of~$G$ and~$H$ into the product~$G\times H$.
\end{proof}


\subsection{The Burnside ring of crossed actions}

We are now in a situation where we can circumvent the lack of good functorial properties of the crossed Burnside ring construction~$G\mapsto\rmB^\times(G)$ by defining a global version, the Burnside ring~$\rmB(\calX)$ of crossed actions.

The abelian group~$\rmB(\calX)$ is generated by elements~$b(\delta\colon X\to G)$, one for each finite crossed action~\hbox{$\delta\colon X\to G$}, subject to the relations
\[
b(\delta\colon X\to G)=b(\epsilon\colon Y\to H)
\]
whenever~$\delta\colon X\to G$ and~$\epsilon\colon Y\to H$ are equivalent~(in the sense we defined in Definition~\ref{def:equivalence}), and
\[
b((\delta,\epsilon)\colon X\sqcup Y\to G\times H)=b(\delta\colon X\to G)+b(\epsilon\colon Y\to H),
\]
where~$(\delta,\epsilon)\colon X\sqcup Y\to G\times H$ is equivalent to the 
sum of~$\delta$ and~$\epsilon$, as in Proposition~\ref{prop:sum_of_crossed}. The ring structure on~$\rmB(\calX)$ is given by the product~\eqref{eq:product_of_crossed} of crossed actions.

\begin{proposition}\label{prop:crossed_like_racks}
For every finite group~$G$, there is a morphism
\begin{equation}\label{eq:local_to_global}
\rmB^\times(G)\longrightarrow\rmB(\calX)
\end{equation}
of rings. These are jointly surjective, i.e., their direct sum is surjective. \end{proposition}

\begin{proof}
The morphism is given by the identity on representatives, i.e., the class of the crossed~$G$--set~$\delta\colon X\to G$ is sent to the class of the crossed action~\hbox{$\delta\colon X\to G$}. We need to check that this respects the addition. If~\hbox{$Y\to G$} and~\hbox{$Z\to G$} are two crossed~$G$--sets, then their sum, as a crossed~$G$--set is~$Y\sqcup Z\to G$, whereas it is~\hbox{$Y\sqcup Z\to G\times G$} as a crossed action. 

\begin{lemma}
The crossed actions~$Y\sqcup Z\to G$ and~\hbox{$Y\sqcup Z\to G\times G$} are equivalent.
\end{lemma}

\begin{proof}
An equivalence in the sense of Definition~\ref{def:equivalence} is given by the identity on underlying sets, and the diagonal~$G\to G\times G$ morphism of groups.
\end{proof}

Joint surjectivity of the morphisms~\eqref{eq:local_to_global} now follows from the obvious fact that every finite crossed action can be presented as a crossed~$G$--set for some finite group~$G$. 
\end{proof}

The reader will note that the structure we revealed in~$\rmB(\calX)$ by Proposition~\ref{prop:crossed_like_racks} is very much resemblant to the structure of~$\rmB(\calR)$ established earlier in~Propositions~\ref{prop:crossed_to_Burnside_of_racks} and~\ref{prop:jointly_surjective}. The following result explains this observation.

\begin{theorem}\label{thm:iso_BX_BR}
There is an isomorphism
\[
\rmB(\calX)\cong\rmB(\calR)
\]
of rings.
\end{theorem}

\begin{proof}
We start by defining a morphism~$\rmB(\calX)\to\rmB(\calR)$, i.e., an additive invariant of finite crossed actions with values in~$\rmB(\calR)$. We claim that this is given by sending~$\delta\colon X\to G$ to the class~$b(X)$ of the rack~$X$ with~\hbox{$x\rhd y=\delta(x)y$}. To see that this is an additive invariant, note that equivalent crossed actions give rise to isomorphic racks and that the rack corresponding to a sum~\hbox{$(\delta,\epsilon)\colon X\sqcup Y\to G\times H$} is decomposed into~$X$ and~$Y$ so that it maps to the sum~$b(X)+b(Y)$ in~$\rmB(\calR)$. We note that this homomorphism is multiplicative.

We now define a morphism~$\rmB(\calR)\to\rmB(\calX)$, i.e., an additive invariant of finite racks with values in~$\rmB(\calX)$. We claim that this is given by sending a finite rack~$(R,\delta)$ to~$b(\delta\colon R\to\Aut(R))$. To see that this function is an additive invariant, note first that isomorphic racks give isomorphic crossed actions. Then, a decomposition of a rack~$R$ into a disjoint union of subracks~$S$ and~$T$ gives rise to equivalent crossed actions~$R\to\Aut(R)$ and~\hbox{$S\sqcup T\to\Aut(S)\times\Aut(T)$}. This time, the equivalence is given by the identity on underlying sets, and the inclusion
\[
\Aut(S)\times\Aut(T)\leqslant\Aut(S\sqcup T)=\Aut(R)
\]
of groups.

Finally, we need to check that the compositions are the identity. This is trivial for the composition~$\rmB(\calR)\to\rmB(\calX)\to\rmB(\calR)$, as this sends the class~$b(R,\rhd)$ of a rack~$(R,\rhd)$ to itself. The other composition sends a crossed action~\hbox{$\delta\colon X\to G$} to the crossed action~\hbox{$\delta\colon X\to \Aut(X)$}, and the latter is equivalent to the former via the equivalence that is given by the identity on underlying sets, and the morphism~$G\leqslant\Aut(X)$ of groups given by the action.
\end{proof}


\begin{remark}
In the proof, it might be tempting to try and use the canonical embedding~\hbox{$R\to\rmF(R)$} into the free group as a crossing. However, there is no reason for it to be equivariant; it only becomes equivariant when we pass to the quotient by the relations that define the enveloping group~$\Gr(R,\rhd)$, which turns out to be the initial object among the possible crossings. However, it is almost never finite, only when~$R=\emptyset$. On the other hand, the automorphism group~$\Aut(R)$ is always finite when the rack~$R$ is.
\end{remark}



\vfill

School of Mathematical Sciences, Lancaster University, Lancaster
LA1 4YF, UNITED KINGDOM\\
\href{mailto:n.mazza@lancaster.ac.uk}{n.mazza@lancaster.ac.uk}

\vspace{\baselineskip}

School of Mathematical and Physical Sciences, University of Sheffield, Sheffield S3 7RH, UNITED KINGDOM\\
\href{mailto:m.szymik@sheffield.ac.uk}{m.szymik@sheffield.ac.uk}\\
Department of Mathematical Sciences, NTNU Norwegian University of Science and Technology, 7491 Trondheim, NORWAY\\
\href{mailto:markus.szymik@ntnu.no}{markus.szymik@ntnu.no}


\end{document}